\documentclass[11pt]{article}
\usepackage[top=1in, bottom=1in, left=1in, right=1in]{geometry}

\usepackage{./dsm5_header}

\setTitle{Time Dependent Scattering from a Grating}

\usepackage{footnote} 

\addAuthor{Li fan\footnote[1]{Department of Mathematics, University of Delaware, 210 S. College Ave., Newark, Delaware}}
\addAuthor{Peter Monk\footnotemark[1]}
\usepackage{amssymb,amsmath,amsfonts}
\usepackage{lineno,color}
\modulolinenumbers[5]
\usepackage{tikz}         

\def\stack#1#2#3{\rlap{#1}\lower#3\hbox{#2}}

\def\bfx{\mathbf{x}}
\def\bfd{\mathbf{d}}
\def\bfe{\mathbf{e}}

\newtheorem{theorem}{Theorem}
\newtheorem{lemma}{Lemma}
\newtheorem{remark}{Remark}
\newtheorem{assumption}{Assumption}
\def\hw{\hat{w}}

\def\red{}
\begin{document}
\placeTitlePage

\begin{outline}
Computing the electromagnetic field due to a periodic grating is critical for
assessing the performance of thin film solar voltaic devices.  In this paper we
investigate the computation of these fields in the time domain (similar problems also arise in simulating antennas).  Assuming a translation invariant periodic grating this reduces to solving the wave equation in a periodic domain. Materials used in practical devices have frequency dependent coefficients, and we
provide a first proof of existence and uniqueness for a general class of such materials.
Using Convolution Quadrature we can then prove time stepping error estimates.  We end with some preliminary numerical results that demonstrate the convergence and stability of the scheme.
\end{outline}
\tableofcontents

\section{Introduction}
We wish to approximate time domain electromagnetic scattering from a  periodic grating.  We shall assume that the grating is translation invariant in one direction, so that 
Maxwell's equations can simplified to obtain two Helmholtz equations governing the s- and p-polarized waves.
Our intended application is to modeling solar voltaic devices.  Usually such devices are modeled in the frequency domain, and for descriptions of frequency domain applications in this area see \cite{jin14,cai15}.  We will consider the problem in the time domain with the potential benefit of being able to compute results at a range of frequencies in one simulation although we do not investigate that aspect here. \red{Note that although our interest is in periodic gratings, similar problems also arise in antenna theory \cite{ril08} and \cite[Section10.2.2]{cai15}. We expect that the theory developed here can be extended to that case.}

We start by describing the problem assuming that the grating is translation invariant parallel to the $y$ axis, \red{so that the permittivity of the material in the grating is independent of $y$ (the magnetic permeability is assumed to be that of free space).}
Because of the reduction in dimension afforded by translation invariance\red{, Maxwell's equations can be reduced to the wave equation in two spatial dimensions (see for example \cite[Section 5.1]{cakoni-colton}). So} we assume that the electromagnetic wave is described by a scalar function $u=u(\bfx,t)$
depending on position $\bfx=(x,z)\in \mathbb{R}^2$ and time $t>0$ that satisfies
\begin{equation}
\frac{1}{c^2}b*\frac{\partial^2 u}{\partial t^2}=\nabla\cdot(a*\grad u)
\mbox{ in }\mathbb{R}^2,\;t>0.
\label{wave}
\end{equation}
Here $c>0$ is the speed of light in vacuum, the symbol $*$ denotes convolution in time and the functions $a$ and $b$ describe the
medium in which the electromagnetic field propagates. Two choices are of interest:
\begin{enumerate}
\item The choice 
\[
a=\delta(t)\mbox{ and } b=\epsilon_r
\]
where $\delta$ is the Dirac delta and $\epsilon_r$ is the time domain relative permittivity of the medium.  In this case the wave is said to be s-polarized and 
$u$ represents the $y$ component of the electric field.  
\item Alternatively 
\[
a=1/\epsilon_r\mbox{ and } b=\delta(t)
\]
in which case the field is said to  p-polarized and $u$ represent the $y$ component of the magnetic field.
\end{enumerate}
Often, for simplicity, it is assumed that $a$ and $b$ are independent of frequency and are real, bounded and uniformly positive piecewise continuously differentiable functions of position. However for
realistic materials both $a=a(\bfx,t)$ and $b=b(\bfx,t)$.  

Since the medium is assumed to be a grating, there is a period $L>0$ such that
\begin{equation*}
\left.\begin{array}{rcl}
a(x+L,z,t)&=&a(x,z,t)\\
b(x+L,z,t)&=&b(x,z,t)\end{array}\right\}\mbox{ for all }\mathbf{x}=(x,z)\in \mathbb{R}^2\mbox{ and }t\in \mathbb{R}.
\end{equation*}
In addition the grating is assumed to have a finite height $H$ such that
\begin{equation}
a=b=\delta\mbox{ for all }x\in\mathbb{R}\mbox{ and } z<0\mbox{ or } z>H.\label{ab}
\end{equation}
This assumption can easily be relaxed to allow for different materials above and below the cell (one of our examples features this). 
We postpone further discussion of the assumptions regarding these coefficients until we have introduced sufficient notation.

Later we will reduce the problem to a bounded domain called the ``unit cell'' defined by
\[
\Omega=(0,L)\times (0,H),
\]
and we will also need the unbounded strip 
\[
S=(0,L)\times\mathbb{R}.
\]
We assume that the total field $u$ is due to an incident plane wave $u^i$ propagating towards the bottom $y=0$ of the grating.  In particular
\[
u^i(\bfx,t)=f(t-\bfd\cdot\bfx/c), \quad \bfx\in\mathbb{R}^2,
\]
for some twice continuously differentiable function $f$, and  unit vector
\[
\bfd=\red{(d_1,d_2):=}({\rm{}cos}s\alpha,{\rm{}sin}\alpha),\quad 0<\alpha<\pi.
\]
So $\alpha=\pi/2$ gives a plane wave propagating up along the $z$-axis at normal incidence to the
grating. In general, by linearity, the total field can be written as
\[
u=u^i+u^s
\]
where $u^s$ is an unknown scattered field to be determined.

Notice that the incident field $u^i$ is not periodic in $x$ but instead
\begin{eqnarray}
u^i(x+L,z,t)&=&f(t-d_1(x+L)/c-\red{d_2}z/c)=f(t-d_1L/c-\bfd\cdot\bfx/c)\nonumber\\&=&u^i(x,z,t-d_1L/c)\label{rpc}
\end{eqnarray}
for any $x,z$ and $t$.  Thus we expect that the scattered field and hence the total field $u$ should have the same translation 
properties so we impose
\[
u(x+L,z,t)=u(x,z,t-d_1L/c) \mbox{ for all } x,z\mbox{ and } t.
\]
This is the time domain counterpart of quasi-periodicity in the frequency domain~\cite{wilcox_book}.

To avoid the retarded periodicity condition (\ref{rpc})  it is common \cite{veys93,mathis_phd} to change variables
for this problem and define
\[
w(x,z,t)=u(x,z,t+(x-L)d_1/c), \mbox{ and } w^i(x,z,t)=u^i(x,z,t+(x-L)d_1/c) \red{.}
\]
With this definition it is clear that $w$ and $w^i$ (and hence $w^s$ defined in the same way) is periodic in $x$ since
\begin{eqnarray*}
w(x+L,z,t)&=&u(x+L,z,t+xd_1/c)=u(x,z,t+(x-L)d_1/c)\\&=&w(x,z,t).
\end{eqnarray*}
The incident field becomes the \red{trivialy} $x$-periodic field \red{(independent of $x$)}
\begin{eqnarray}
w^i(x,z,t)&=&u^i(x,z,t+(x-L)d_1/c)=f(t-Ld_1/c-d_2z/c).\label{winc}
\end{eqnarray}
We assume the field in the unit cell is quiescent before $t=0$ so that $w^s=0$ for
$t<0$.  This requires that $w^i=0$ on $\Omega$ for $t<0$ or
\[
f(t-Ld_1/c-d_2z/c)=0
\]
for $t<0$ and $0<z<H$.  Assuming $d_1\geq 0$ and $d_2>0$ it suffices that $f(t)=0$ for $t<0$.

With the above change of variables \red{equation (\ref{wave})} become more complicated.  Setting $\bfe_1=(1,0)$ we have
\begin{eqnarray*}
\grad u&=&\grad w-\frac{d_1}{c}\frac{\partial w}{\partial t}\bfe_1,\\
\Delta u&=&\Delta w-2\frac{d_1}{c}\frac{\partial^2w}{\partial t\partial x}+\frac{d_1^2}{c^2}\frac{\partial^2w}{\partial t^2}.
\end{eqnarray*}
So since $\nabla\cdot a\grad u=a\Delta u+\grad a\cdot\grad u$ we conclude that $w$ satisfies
\begin{equation}
\left(\frac{b-ad_1^2}{c^2}\right)*\frac{\partial^2 w}{\partial t^2}=\nabla\cdot \red{(a*\grad w)}
-\frac{d_1}{c}\frac{\partial}{\partial x}\left(a*\frac{\partial w}{\partial t}\right)-\frac{d_1}{c}a*\frac{\partial^2w}{\partial t\partial x}
\label{weqn}
\end{equation}
for all $x,z$ and $t$.  The problem we wish to solve is to compute $w$ such that (\ref{weqn}) holds together with the splitting
\[
w=w^i+w^s\mbox{ in } S\times\mathbb{R},
\]
together with the initial conditions 
\[
w^s=\frac{\partial w^s}{\partial t}=0\mbox{ in }S\mbox{ at }t=0,
\]
and in addition $w$ is $L$-periodic in $x$.

The change of variables we have used above is well-known in the engineering literature
\cite{veys93} where an explicit finite difference time domain (FDTD) technique is used \red{for} discretization and  time stepping.  In this case the time step must be chosen depending on the angle of incidence. Some later authors \cite{ril08,jin_riley_book} have use implicit schemes to avoid stability restrictions. We shall use implicit methods here.
Another  interesting FDTD paper for gratings that discusses, amongst other things, locally refined grids is \cite{hol02}. The numerical analysis of several models of dispersive media are discussed in \cite{jli10,jli08} and in particular \cite{jli_book}, but these studies do not cover grating problems, and do not provide an analysis for a general class of problems.

Existence and uniqueness questions for gratings were studied in \cite{mathis_phd}.  In that thesis,
existence and uniqueness is proved using a weighted norm in time for frequency independent coefficients. This is a special case of our result which improves the norm and also allows frequency dependent coefficients.  We use the Laplace
transform as a tool and in addition prove error estimates for a limited class of time-stepping schemes.  We also provide preliminary numerical results verifying the temporal convergence rate of the scheme, as well as showing results for two frequency dependent 
models covered by our theory.

The layout of the paper is as follows.  In the next section we prove existence and uniqueness of a 
solution to the time domain problem.  We also give error estimates for \red{Backward Euler and Backward Differentiation Formula 2 (BDF2)
based discretization in time. Both these time stepping rules are implicit, the former being first order and the latter second order \cite{iserles}.} In Section \ref{space} we show how we discretized in space and truncated the problem using a domain decomposition strategy. In Section \ref{numer} we give three numerical examples: we first verify the predicted convergence rate in a simple case, then we give two examples of computations using standard frequency dependent coefficient models.  Finally in Section~\ref{concl} we 
put the study into perspective.

To simplify the presentation, for the remainder of this paper we will assume s-polarization so that $a=\delta$.  \red{The p-polarization case can be analyzed in a similar way.  In that case $b=\delta$ and we need to required that $1-d_1^2$ is strictly positive.  In addition, we need to ensure coercivity of the appropriate bilinear form.  The main assumption is that if $\hat{a}(s)$is the Laplace transform of $a$ with parameter $s$ (see upcoming (\ref{lapdef}))  and  if $\sigma=\Re(s)>0$ then there is a constant $C$ depending on $\sigma$ such that
\[
\Re(\overline{s}\hat{a}(s))\geq C.
\]
The corresponding assumption for s-polarization is given in detail Assumption \ref{A1}, and we would also need to assume the differentiability
conditions on $\hat{a}$.  Numerical examples of the p-polarized case are not investigated here.}

\section{Existence and uniqueness}
For theoretical purposes it is convenient to work with the scattered field.  So setting
\[
w=w^i+w^s
\]
and noting that $w^i$ satisfies (\ref{weqn}) with $a=b=\delta$ we have that $w^s$ satisfies
\begin{eqnarray}
\left(\frac{b-d_1^2\delta}{c^2}\right)*\frac{\partial^2 w^s}{\partial t^2}&=&\Delta w^s
-2\frac{d_1}{c}\frac{\partial^2w^s}{\partial t\partial x}+F\mbox{ in }S\times \mathbb{R}_+\label{wsteqn}\\
w^s&=&0\mbox{ in }S\mbox{ at }t=0\\
\frac{\partial w^s}{\partial t}&=&0\mbox{ in }S\mbox{ at }t=0\\
w^s(L,z,t)&=&w^s(0,z,t)\mbox{ for }t>0,\;z\in\mathbb{R},\label{ubc}\\
\frac{\partial w^s}{\partial x}(L,z,t)&=&\frac{\partial w^s}{\partial x}(0,z,t)\mbox{ for }t>0,\;z\in\mathbb{R},\label{diffbc}
\end{eqnarray}
where 
\[
F=\left(\frac{\delta(t)-b}{c^2}\right)*\frac{\partial^2 w^i}{\partial t^2}.
\]
Notice that, \red{by our definition of $H$ (see (\ref{ab})),} $F=0$ if $z>H$ or $z<0$ so  $F$  has support in $\Omega$.  Our assumptions
on $f$ also imply that provided $b$ is causal, $F=0$ in $\Omega$ for $t<0$.  We can extend $w^s$ by zero to time $t<0$ and
hence obtain a causal function defined for all $t$.

To analyze this problem we use the Laplace transform in time \cite{lub94,BHa86}.  For any sufficiently smooth
function $g=g(t)$  with at most exponential growth for large time, the Laplace transform is
$\hat{g}=\hat{g}(s)$ given by
\begin{equation}
\hat{g}(s)={\cal L}(g)(s)=\int_0^\infty g(t)\exp(-st)\,dt,
\label{lapdef}
\end{equation}
where we choose the transform variable to be $s=\sigma-i\omega$ for
$\sigma\in \mathbb{R}$ and $\sigma>0$, and $\omega\in\mathbb{R}$. Then because of our initial conditions
\[
{\cal L}\left(\frac{\partial^2 w^s}{\partial t^2}\right)=s^2\hat{w}^s\mbox{ and }
{\cal L}\left(\frac{\partial^2 w^s}{\partial t\partial x}\right)=s\frac{\partial \hat{w}^s}{\partial x}.
\]
In addition
\begin{eqnarray*}
\hat{w}^i(x,z)&=&\int_0^\infty\exp(-st)f(t-Ld_1/c-d_2z/c)\,ds\\
&=& \int_{-Ld_1/c-d_2z/c}^\infty \exp(-s(\tau+Ld_1/c+d_2z/c))f(\tau)\,d\tau\\
&=&\hat{f}(s)\exp(-sLd_1/c)\exp(-sd_2z/c).
\end{eqnarray*}
As expected this is a scaled plane wave in the Laplace domain \red{that is actually independent of $x$ (see equation (\ref{winc})).}

Formally taking the Laplace transform of (\ref{wsteqn}) we are lead to seek  $\hat{w}^s\in H^1_{\rm{}p}(S)$ that satisfies
\begin{eqnarray*}
s^2\left(\frac{\hat{b}-d_1^2}{c^2}\right)\hat{w}^s&=&\Delta \hat{w}^s
-2s\frac{d_1}{c}\frac{\partial\hat{w}^s}{\partial x}+\hat{F}\mbox{ in }S,\\
\hat{w}^s(L,z)&=&\hat{w}^s(0,z)\mbox{ for }y\in\mathbb{R},\\
\frac{\partial \hat{w}^s}{\partial x}(L,z)&=&\frac{\partial \hat{w}^s}{\partial x}(0,z)\mbox{ for }z\in\mathbb{R}.\end{eqnarray*}
Here $\hat{F}=s^2\red{(1-\hat{b})}\hat{w}^i/c^2$ and $\hat{b}$ is the Laplace transform of $b$.  

To formulate a variational problem for this system,  let
\[
H^1_{\rm{}p}(S)=\left\{ f\in H^1(S)\;|\;f(L,z)=f(0,z)\mbox{ for all }z\in\mathbb{R}\right\}.
\]
\red{(where the subscript p recalls the $x$-periodicity of the functions)} with the $s$-dependent norm
\[
\Vert\hat{w}^s\Vert_{H^1_{\rm{}p}(S)}^2= \int_{S}\left[|\grad \hat{w}^s|^2+\left(\frac{|s|^2}{c^2}\right)
|\hat{w}^s|^2\right]\,dA
\]
\red{where $dA=dx\,dz$.} We shall  require that $\hat{w}^s\in H^1_{\rm{}p}(S)$  and this requirement replaces a radiation condition. It holds because
$\Re(s)=\sigma>0$.

We can now write a Galerkin formulation for the Laplace transformed problem  as usual by multiplying by the complex conjugate of a  test function and integrating by parts (using the periodicity of the normal derivative to cancel terms on $x=0$ and $x=L$). We seek $\hat{w}^s\in H^1_{\rm{}p}(S)$ such that
\begin{equation}
a(\hat{w}^s,\xi)=\int_{S}\hat{F}\overline{\xi}\,dA\mbox{ for all }
\xi\in  H^1_{\rm{}p}(S),\label{wsweak}
\end{equation}
where the over-bar denotes complex conjugation and
\[
a(\hat{w}^s,\xi)=\int_{S}\left[\grad \hat{w}^s\cdot\grad \overline{\xi}+s^2\left(\frac{\hat{b}-d_1^2}{c^2}\right)
\hat{w}^s\overline{\xi}+2s\frac{d_1}{c}\frac{\partial \hat{w}^s}{\partial x}\overline{\xi}\right]\,dA.
\]
At this stage we need to specify our assumptions on the frequency dependent coefficient $\hat{b}$.
We assume
\begin{assumption}\label{A1}
The coefficient $\hat{b}=\hat{b}(\bfx,s)$ is piecewise continuously differentiable in $\bfx$. In addition \begin{enumerate}
\item For almost every $\bfx \in S$ the coefficient $\hat{b}$ is analytic in $s$ for $\Re(s)>\sigma_0>0$ for any $\sigma_0$, and bounded independent of $s$ (but the bound may depend on $\sigma_0$).
\item There is a constant $\gamma_0$ such that
\[
\Re(s(\hat{b}(\bfx,s)-d_1^2))>\sigma \gamma_0>0
\]
for $\Re(s)=\sigma>0$ and all $\bfx\in S$.
\item $\hat{b}(x,z,s)=1$ for $z<0$ or $z>H$ and all $x$, $s$.
\end{enumerate}
\end{assumption}
In Section \ref{freq_dep} we will give two important examples of a frequency dependent coefficient satisfying the above assumptions.

Now we can use the Lax-Milgram Lemma to prove existence of a solution. First we verify coercivity and continuity.
\begin{lemma} Suppose that   $\hat{b}$ satisfies the Assumption \ref{A1}, and $\Re(s)=\sigma>\sigma_0>0$ for some $\sigma_0$. Then
the sesquilinear form $a(\cdot,s\cdot)$ is coercive and bounded.  In particular for every $\hat{v}\in  H^1_{\rm{}p}(S)$
\[
|a(\hat{v},s\hat{v})|\geq\sigma\min(1,\gamma_0)\Vert\hat{v}\Vert^2_{H^1_{\rm{}p}(S)}.
\]
In addition there is a constant $C$  depending on $\sigma_0$ but  independent of $s$, $\hat{u},\hat{v}\in H^1_{\rm{}p}(S)$ such that
\[
|a(\hat{u},\hat{v})|\leq C\Vert\hat{u}\Vert_{H^1_{\rm{}p}(S)}
\Vert\hat{v}\Vert_{H^1_{\rm{}p}(S)}.
\]
\end{lemma}
\begin{remark} The dependence of the constant in the boundedness estimate above on $\sigma_0$ comes from
the frequency dependence of $\hat{b}$.  If $\hat{b}$ is frequency independent, then the constant is independent of $\sigma_0$.
\end{remark}
\begin{proof}
 Using the by now standard trick
of Bamberger and Ha Duong \cite{BHa86} we choose $\xi=s\hat{v}$ and obtain
\[
a(\hat{v},s\hat{v})=\int_{S}\left[\overline{s}\,|\grad\hat{v}|^2+s|s|^2\left(\frac{\hat{b}-d_1^2}{c^2}\right)
|\hat{v}|^2+2|s|^2\frac{d_1}{c}\frac{\partial \hat{v}}{\partial x}\overline{\hat{v}}\right]\,dA.
\]
But since integration by parts in $x$ shows that
\[
\int_{S}\frac{\partial \hat{v}}{\partial x}\overline{\hat{v}}+\hat{v}\frac{\partial \overline{ \hat{v}}}{\partial x}\,dA=0
\]
we have 
\[
\Re\left(\int_{S}\frac{\partial \hat{v}}{\partial x}\overline{\hat{v}}\,dA\right)=0
\]
and
so we have proved coercivity because, using our assumption on $\hat{b}$,
\begin{eqnarray*}
\Re [ a(\hat{v},s\hat{v})]&=&\red{\int_{S}\Re\left[\overline{s}\,|\grad\hat{v}|^2+|s|^2s\left(\frac{\hat{b}-d_1^2}{c^2}\right)
|\hat{v}|^2\right]\,dA}\\
&\geq& \sigma\min(1,\gamma_0)\Vert\hat{v}\Vert^2_{H^1_{\rm{}p}(S)}.
\end{eqnarray*}

In addition $a(\cdot,\cdot)$ is bounded because for any $\hat{u},\hat{v}\in H^1_{\rm{}p}(S)$ we have
\begin{eqnarray*}
|a(\hat{u},\hat{v})|&\leq&\Vert \grad \hat{u}\Vert_{L^2(S)}\Vert \grad \hat{v}\Vert_{L^2(S)}+\frac{|s|^2}{c^2}\Vert \hat{ b}-d_1^2\Vert_{L^{\infty}(S)}\Vert \hat{u}\Vert_{L^2(S)}\Vert \hat{v}\Vert_{L^2(S)}\\&&\qquad +2|s|\frac{d_1}{c}\Vert\grad \hat{u}\Vert_{L^2(S)}\Vert \hat{v}\Vert_{L^2(S)}.
\end{eqnarray*}
Use of our assumption on the coefficient $\hat{b}$ and standard inequalities demonstrates the required continuity.
\end{proof}
Using the above lemma and the Lax-Milgram Lemma we have the following result
\begin{theorem}
Assume that $\hat{b}$ satisfies Assumption~\ref{A1}  and $\Re(s)=\sigma>0$.  Then problem (\ref{wsweak}) has a unique solution and 
\[
 \sigma\min(1,\gamma_0)\Vert \hat{w}^s\Vert_{H^1_{\rm{}p}(S)}\leq \frac{1}{c^2}\Vert \hat{b}-1\Vert_{L^\infty(S)}\Vert \hat{F}\Vert_{L^2(S)}
 \]
 \end{theorem}
 
 This verifies the existence of the solution in the Laplace domain and this in turn  gives a time domain existence and uniqueness theorem.  Note that we need the condition $\Re(s(\hat{b}(\bfx)-d_1^2))\geq\sigma\gamma_0>0$ for all $\bfx$  in $S$. Since $d_1={\rm{}sin}\theta$ where $\theta$ is the angle of incidence, this may limit the magnitude of the angle of incidence.
  
Using Lubich's theory of convolution quadrature \cite{lub94}, and considering a finite time period $[0,T]$ for some $T>0$ we can now state the following existence and uniqueness result.  The usual space for this theory is
\[
H^m_0((0,T),H^1_{\rm{}p}(S))=\{g|_{(0,T)}\;|\; g\in H^m(\mathbb{R},H^1_{\rm{}p}(S)) \mbox{ with }g=0 \mbox{ for }t<0\},
\]
which places compatibility conditions on the data at $t=0$ if $m\geq 1$.
\begin{theorem}  Suppose $F\in H^m_0((0,T),L^2(S))$ for some $m\geq 0$.  Then there exists a unique weak solution $w^s\in H^m_0((0,T),H^1_{\rm{}p}(S))$ of the time domain problem
(\ref{wsteqn})-(\ref{diffbc}).\end{theorem}
\begin{proof} This is an application of a slight generalization of \cite[Lemma 2.1]{lub94}
(for similar analysis see Section 2.2 of the same paper).  In the transform domain we can write
\[
\hat{w}^s=K(s) \hat{F}
\]
and the solution operator $K(s):L^2(\Omega)\to H^1_{\rm{}p}(\Omega)$ satisfies, for any fixed $\sigma_0>0$ and all $s$ with $\Re(s)>\sigma_0$
\[
\Vert K(s)\Vert_{L^2:H^1}\leq C
\]
for some $C$ depending on $\sigma_0$ but independent of $s$.  Here $\Vert \cdot\Vert_{L^2:H^1}$ is the operator norm for maps from $L^2(S)\to H^1_{\rm{}p}(S)$.
Parseval's theorem gives, using the contour $\Re(s)=\sigma>\sigma_0$,
\[
\Vert \exp(-\sigma t)w^s\Vert_{H^m((0,T),H^1_{\rm{}p}(S))}\leq C
\Vert \exp(-\sigma t)F\Vert_{H^m((0,T),L^2(S))}.
\]
\end{proof}
To obtain results for time discretization we can appeal to Lubich's theory of Convolution Quadrature~\cite{lub94}.  This is simplest for a multistep scheme which we now describe (an implicit Runge-Kutta scheme could also be used but is beyond the scope of the paper).
Let $\Delta t>0$ denote the timestep and let $t_n=n\Delta t$, $n\geq 0$.  Given $g(t,y)$,  consider the differential equation $y'=g(t,y)$  for $t>0$ with $y(0)=0$, then a general $k$-step multistep scheme applied to this ordinary differential equation is to find $\{y_n\}_{n=0}^\infty$ such that
\[
\sum_{j=0}^k\alpha_jy_{n-j}=\Delta t\sum_{j=0}^k\beta_jg(t_{n-j},y_{n-j})
\]
where we assume compatible initial data and set $y_j=0$ for $j\leq 0$.  The coefficients $\{\alpha_j,\beta_j\}_{j=0}^k$ define the method and we assume that $\alpha_0/\beta_0>0$.  Then following Lubich~\cite{lub94}, let
\[
\gamma(\zeta)=\frac{\sum_{j=0}^k\alpha_j\zeta^j}{\sum_{j=0}^k\beta_j\zeta^j}, \quad\zeta\in\mathbb{C}.
\]
 Lubich shows Convolution Quadrature time stepping is equivalent to the following parameterized problem.  Let
\[
W^s(\bfx)=\sum_{j=0}w^s_n(\bfx)\zeta^n
\]
for $\zeta\in\mathbb{C}$ small enough.  Here we \red{shall prove that $w^s_n$ converges to $w^s(\cdot,t_n)$ as $\Delta t\to 0$,} and we take $w^s_n=0$ for $n\leq 0$.  Then 
 $W^s\in H^1_{\rm{}p}(S)$ satisfies
 \begin{eqnarray}
a_{\Delta t}(W^s,\xi)=\int_{S}\hat{F}_{\Delta t}\overline{\xi}\,dA\mbox{ for all }
\xi\in  H^1_{\rm{}p}(S),\label{wsweakz}
\end{eqnarray}
for all $|\zeta|<1$, where
\begin{eqnarray*}
a_{\Delta t}(W^s,\xi)&=&\int_{S}\left[\grad W^s\cdot\grad \overline{\xi}+\left(\frac{\gamma(\zeta)}{\Delta t}\right)^2\left(\frac{b_{\Delta t}-d_1^2}{c^2}\right)
W^s\overline{\xi}\right.\\&&\qquad\left. +2\frac{\gamma(\zeta)}{\Delta t}\frac{d_1}{c}\frac{\partial W^s}{\partial x}\overline{\xi}\right]\,dA,
\end{eqnarray*}
and $b_{\Delta t}$ and $F_{\Delta t}$ are obtained from $\hat{b}$ and $\hat{F}$ by replacing the transform parameter $s$ by $\gamma(\zeta)/\Delta t$.

To obtain a time stepping scheme it then suffices to equate terms in $\zeta^n$ on the left and right hand side of (\ref{wsweakz}) to obtain equations
for $w_n^s$ in terms of previous values.  When $\hat{b}$ is frequency independent, we obtain the usual multistep update formula that can be derived alternatively by applying the multistep method to the standard weak form of (\ref{wsteqn}) (however the approach we have given will provide an error analysis).  

Obviously this scheme is not yet computable because $S$ is unbounded.  One possibility
is to rewrite ({\ref{wsweakz}) by truncating $S$ using auxiliary boundaries above and below the grating, then using an integral equation or Dirichlet-to-Neumann map on the auxiliary boundary to close the system.  We discuss an approach similar to the Dirichet-to-Neumann map approach based on upwind transmission conditions in the next section.  

The theory of convolution quadrature provides an error estimate for this problem.  Let $(\partial_{\Delta t} W^s)_n$ denote the coefficient of $\zeta^n$
in the expansion of $\gamma(\zeta)W^s/\Delta t$.  This is the finite difference approximation of the time derivative corresponding to the 
given multistep method.  We have the following theorem:

\begin{theorem}\label{tm_conv}
Suppose the multi-step method is A-stable, order $p$ convergent and that $\gamma(\zeta)$ has no poles on the unit circle.  In addition suppose $b$ satisfies Assumption \ref{A1},  that $m=p+2$ and that $\partial^j F/\partial t^j=0$ at $t=0$ for $j=0,1,\cdots,m-1$. Then
for $0\leq t_n\leq T$ we have
\begin{eqnarray*}
&&\Vert \nabla (w^s(.,t_n)-w^s_n)\Vert_{L^2(\Omega)}+\frac{1}{c}\Vert \frac{\partial}{\partial t} w^s(.,t_n)-(\partial_{\Delta t} W^s)_n)\Vert_{L^2(\Omega)}\\
&\leq &C(\Delta t)^p\int_0^t\Vert\frac{\partial^m F}{\partial t^m}(.,\tau)\Vert_{L^2(\Omega)}\,d\tau
\end{eqnarray*}
where $C$ is independent of $\Delta t$, $w^s$ and the discrete solution, but depends on $T$.
\end{theorem}
\begin{remark} The canonical examples of suitable time stepping schemes that satisfy the requirements of the theorem are backward Euler ($p=1$) and BDF2 ($p=2$).  In the latter case $m=4$ and we need $F$ to have $4$ weak derivatives in time and satisfy the compatibility conditions
$\partial^j F/\partial t^j=0$ at $t=0$ for $j=0,1,\cdots,3$.
\end{remark}

\begin{proof}  The proof is a direct consequence of Theorem 3.1 (see also the discussion before Corollary 4.2) of~\cite{lub94}.
\end{proof}

\section{Discretization in space}\label{space}
We can derive a time stepping scheme by truncating (\ref{wsweakz}}) using integral equations or the Dirichlet-to-Neumann (DtN) map and then equating terms in $\zeta^n$.  However in order to demonstrate the viability of the approach we will use the Laplace domain approach of Banjai and Sauter \cite{Banjai_2008}.  This requires to solve the Laplace domain problem for several
different choices of the transform parameter $s$.  The algorithm is exactly as in Banjai and Sauter's paper and so we only give details of the Laplace domain problem that we solve. \red{As pointed out by Banjai and Sauter, this approach is trivially parallelizable, although we have not investigated this or other algorithmic improvements here. This can help improve the elapsed time for running the algorithm.}
 
First we make a minor change and will solve for the total field in $\Omega$ (this avoids evaluating $F$ everywhere in the grating, but is less convenient for analysis).
Let $c$ denote the speed of light below the grating for \red{$z<0$}.    
The $x$ periodic total field $\hw=\hw(x,z)$ satisfies
\begin{equation}
\Delta \hw-2\frac{s}{c} d_1\hw_x-\frac{s^2}{c^2}(\hat{b}-d_1^2)\hw=0\mbox{ in }\Omega,
\label{pde:helmholtz}
\end{equation}

To accommodate one of our examples, we assume a slight generalization of the problem discussed so far.  We assume that $\hat{b}=1$ for $z<0$ and $\hat{b}=\hat{b}_1$  for $z>H$ where $\hat{b}_1$ is spatially constant but could be frequency dependent and so depend on $s$.
In the region $\Omega$ we  have, in general, $\hat{b}\not=1$.  

For $z<0$ we see that $\hw$ satisfies
\begin{equation}
\Delta \hw-2\frac{s}{c}d_1 \hw_x-\frac{s^2}{c^2}(1-d_1^2)\hw=0.
\label{pde:below}
\end{equation}
This field can be decomposed into an incident field given by
\[
\hat{w}^{i}(x,z)=\exp(-s\frac{d_2}{c} z) \,\exp(-s\frac{Ld_1}{c})\, \hat{f}(s)
\]
and scattered field as before so that $\hw=\hw^i+\hw^s$. The scattered field also solves the above differential equation for $z<0$ or $z>H$ and we now derive an equation for the scattered field for $z<0$.  Since $\hw^s$  is periodic
\[
\hw^s(x,z)=\sum_{n\in\mathbb{N}}\tilde{w}_n^s(z)\exp(i2\pi nx/L)
\]
for suitable Fourier coefficients $\{\tilde{w}_n^s\}_{n=0}^\infty$.
Substituting this expansion into (\ref{pde:below}) gives the condition
\[
\red{\frac{d^2}{dz^2}}\tilde{w}_n^s-\left(\frac{s^2}{c^2}+\left(2n\pi/L+i\frac{s}{c}d_1\right)^2\right)\tilde{w}_n^s=0,\quad \red{y<0}.
\]
Let $\kappa_n^s$ be defined by
\[
(\kappa_n^s)^2=\left(\frac{s}{c}\right)^2\left(1+\left(2n\pi c/(Ls)+id_1\right)^2\right)
\]
  We need to
choose the signs of the square root so that we have a decaying solution as $z\to-\infty$.  Then
\begin{equation}
\hw^s(x,z)=\sum_{n\in\mathbb{N}}\hat{w}_n^s \exp(i2\pi nx/L)\exp(\kappa_n^sz), \mbox{ for }z<0.
\label{us:expand}
\end{equation}
for suitable expansion coefficients $\hat{w}_n^s$ and where $\Re(\kappa_n^s)>0$.

For $z>H$,  the transmitted total wave $\hw^t$ satisfies
\begin{equation}
\Delta \hw^t-2\frac{s}{c}d_1 \hw^t_x-\frac{s^2}{c^2}(\hat{b}_1-d_1^2)\hw^t=0
\label{pde:above}
\end{equation}
There is no incident wave and the wave is only transmitted.
Proceeding as before we define
$\kappa_n^t$ by
\[
(\kappa_n^t)^2=\left(\frac{s}{c}\right)^2\left(\hat{b}_1+\left(\frac{2n\pi c}{Ls}+id_1\right)^2\right)
\]
where we  choose again $\Re(\kappa_n^t)>0$.
The transmitted field for $z>H$ is given by
\begin{equation}
\hw^t(x,z)=\sum_{n\in\mathbb{N}}\hat{w}_n^t \exp(i2\pi nx/L)\exp(-\kappa_n^t(z-H)), \mbox{ for }z>H\geq 0
\label{ut:expand}
\end{equation}
for suitable expansion coefficients $\hat{w}_n^t$.

The fields $\hw^s$ for $z<0$, $\hw^t$ for $z>H$ and $\hw$ for $0<z<H$ are related by
continuity conditions across the interfaces at $z=0$ and $y=H$.  At $z=0$ we need
\[
\hw(x,0)=\hw^i(x,0)+w^s(x,0)\mbox{ and } \hw_z(x,0)=\hw_z^i(x,0)+\hw_z^s(x,0),\quad 0<x<L.
\]
At the upper interface $z=H$ we need
\[
\hw(x,H)=\hw^t(x,H)\mbox{ and } \hw_z(x,H)=\hw_z^t(x,H),\quad 0<x<L.
\]

To formulate a problem suitable for spatial discretization we have adopted the one-way wave equation or characteristic equation approach
of \cite{HW_book} which is based on impedance boundary conditions (see also the UWVF \cite{HMM06} and for strongly related methods \cite{gander15,gan07}).

Let $\eta $ be a positive constant at our disposal (in fact we choose $\eta=1$). Given $\lambda_0^-\in L^2(0,L)$ and $\lambda_H^-\in L^2(0,L)$ define $\hw=\hw(\lambda_0^-,\lambda_H^-)\in H^1_{\rm{}p}(\Omega)$ by
\begin{eqnarray*}
\Delta \hw-2\frac{s}{c}d_1 w_x-\frac{s^2}{c^2}(b-d_1^2)w&=&0 \mbox{ in }\Omega,\\
\frac{\partial \hw}{\partial z}+\frac{s}{c}\eta \hw&=&\lambda_H^-\mbox{ for }z=H,\,0<x<L,\\
-\frac{\partial \hw}{\partial z}+\frac{s}{c} \eta \hw&=&\lambda_0^-\mbox{ for }z=0,\,0<x<L.
\end{eqnarray*}
This is easily written as the variational problem of seeking $\hw\in H^1_{\rm{}p}(\Omega)$ such that
\begin{equation}
B(\hw,v)=f(v)\mbox{ for all } v\in H^1_p(\Omega)\label{Bprob}
\end{equation}
where
\begin{eqnarray*}
B(\hw,v)&=&\int_\Omega\left\{\grad \hw\cdot\grad\overline{v}+2\frac{s}{c}d_1\hw_x\overline{v}+\frac{s^2}{c^2}(b-d_1^2)w\overline{v}\right\}\,dA\\&&
+\int_{\Sigma_H} \frac{s}{c}\eta \hw\overline{v}\,ds+\int_{\Sigma_0} \frac{s}{c}\eta \hw\overline{v}\,ds.
\end{eqnarray*}
Here $\Sigma_H=\{(x,H)\;|\;0<x<L\}$ while
$\Sigma_0=\{(x,0)\;|\;0<x<L\}$, and 
\[
f(v)=\int_{\Sigma_H}\lambda_H^-\overline{v}\,ds+\int_{\Sigma_0}\lambda_0^-\overline{v}\,ds.
\]

For $z>H$ define $\hw^t(\lambda_H^+)$ by requiring that
\[
\frac{\partial \hw^t}{\partial z}-\frac{s}{c} \eta \hw^t=\lambda_H^+\mbox{ for }z=H,\,0<x<L,
\]
together with the expansion (\ref{ut:expand}).
Using a trigonometric expansion (noting the periodicity of the solution and its normal derivative)
\[
\lambda_H^+=\sum_{n\in \mathbb{N}}\lambda^+_{H,n}\exp(i2\pi nx/L)
\]
then the Fourier coefficients of the transmitted field are
\[
\hw_n^t=-\lambda^+_{H,n}/(\kappa_n^t+s\eta/c),\quad n\in\mathbb{Z},
\]
and we can see that this field is aways well defined since $\Re s >0$ and $\Re(\kappa_n^t)>0$.
Define $F_H(\lambda_H^+)$ by
\begin{eqnarray*}
F_H(\lambda_H^+) &=& \frac{\partial \hw^t}{\partial z}+\frac{s}{c} \eta \hw^t \mbox{ on }\Sigma_H\\
&=&\sum_{n\in\mathbb{N}}\lambda^+_{H,n}\left(\frac{ \kappa_n^t -(s/c)\eta }{\kappa_n^t+(s/c)\eta}\right)\exp(i2\pi nx/L).
\end{eqnarray*}

Similarly for $z<0$ define $\hw^s(\lambda_0^+)$ by requiring that
\[
-\frac{\partial \hw^s}{\partial z}-\frac{s}{c}\eta \hw^s=\lambda_0^+\mbox{ for }z=0,\,0<x<L,
\]
together with the expansion (\ref{us:expand}).
Suppose 
\[
\lambda_0^+=\sum_{n\in \mathbb{N}}\lambda^+_{0,n}\exp(i2\pi nx/L)
\]
then the $n$th Fourier coefficient of the scattered field is
\[
\hw_n^s=-\lambda^+_{0,n}/(\kappa_n^s+(s/c)\eta).
\]
Define $F_0(\lambda_0^+)$ by
\begin{eqnarray*}
F_0(\lambda_0^+) &=& -\frac{\partial \hw^s}{\partial z}+\frac{s}{c} \eta \hw^s \mbox{ on }\Sigma_0\\
&=&\sum_{n\in\mathbb{N}}\lambda^+_{0,n}\left(\frac{ {\kappa_n^s-(s/c)\eta} }{\kappa_n^s+(s/c)\eta}\right)\exp(i2\pi nx/L).
\end{eqnarray*}

It remains to derive equations for $\lambda_H^{\pm}$ and $\lambda_0^{\pm}$.  This is done by enforcing the transmission conditions in impedance form.  At $z=H$ we require
\begin{eqnarray*}
\frac{\partial \hw}{\partial z}+\frac{s}{c} \eta \hw&=&\frac{\partial \hw^t}{\partial z}+\frac{s}{c}\eta \hw^t\mbox{ at }z=H,\\
\frac{\partial \hw}{\partial z}-\frac{s}{c} \eta \hw&=&\frac{\partial \hw^t}{\partial z}-\frac{s}{c}\eta \hw^t\mbox{ at }z=H.
\end{eqnarray*}
Writing these equations in terms of the unknown functions,
\begin{eqnarray*}
\lambda_H^--2\frac{s}{c} \eta \hw(\lambda_0^-,\lambda_H^-)&=&\lambda_H^+,\\
\lambda_H^-&=&F_H(\lambda_H^+),
\end{eqnarray*}
where we have avoided computing the normal derivative of $u$ by writing 
\[
\frac{\partial \hw}{\partial z}=\lambda_H^--\frac{s}{c} \eta \hw.
\]

The same process can be applied at $z=0$.  We require that
\begin{eqnarray*}
\frac{\partial \hw}{\partial z}+ik \eta u&=&\frac{\partial (\hw^s+\hw^i)}{\partial z}+ik\eta (\hw^s+\hw^i)\mbox{ at }y=0,\\
\frac{\partial \hw}{\partial z}-ik \eta \hw&=&\frac{\partial (\hw^s+\hw^i)}{\partial z}-ik\eta (\hw^s+\hw^i)\mbox{ at }y=0.
\end{eqnarray*}
Let
\[
f_-=-\frac{\partial \hw^i}{\partial z}+\frac{s}{c}\eta \hw^i\mbox{ and }
f_+= -\frac{\partial \hw^i}{\partial z}-\frac{s}{c}\eta \hw^i\mbox{ at }z=0.
\]
This gives the equations
\begin{eqnarray*}
-\lambda_0^-&=&-F_0(\lambda_0^+)-f_-,\\
-\lambda_0^-+2\frac{s}{c}\eta \hw(\lambda_0^-,\lambda_H^-)&=&-\lambda_0^+-f_+.
\end{eqnarray*}
In summary we must find $\lambda_H^{\pm}\in L_2(\Sigma_H)$ and $\lambda_0^{\pm}\in L_2(\Sigma_0)$ such that
\begin{eqnarray*}
\lambda_H^--2\frac{s}{c} \eta \hw(\lambda_0^-,\lambda_H^-)&=&\lambda_H^+,\\
\lambda_H^-&=&F_H(\lambda_H^+),\\
\lambda_0^-&=&F_0(\lambda_0^+)+f_-,\\
\lambda_0^--2\frac{s}{c}\eta \hw(\lambda_0^-,\lambda_H^-)&=&\lambda_0^++f_+.
\end{eqnarray*}

In order to discretize the problem, we expand all boundary functions as a finite Fourier series:
\begin{eqnarray*}
\lambda_H^{\pm,N}&=&\sum_{n=-N}^N\lambda^\pm_{H,n}\exp(i2\pi nx/L),\\
\lambda_0^{\pm,N}&=&\sum_{n=-N}^N\lambda^\pm_{0,n}\exp(i2\pi nx/L),
\end{eqnarray*}
and for ease of notation define $\psi^{H}_n(x)=\exp(i2\pi nx/L)|_{\Sigma_H}$ and $\psi^{0}_n(x)=\exp(i2\pi nx/L)|_{\Sigma_0}$. Then using the inner products
\[
\langle u,v\rangle_H=\int_{\Sigma_H}u\overline{v}\,ds,\quad \langle u,v\rangle_0=\int_{\Sigma_0}u\overline{v}\,ds,
\]
we have the discrete system
\begin{eqnarray*}
\langle \lambda_H^{-,N},\psi_p^H\rangle_H-2\frac{s}{c} \eta \langle \hw(\lambda_0^{-,N},\lambda_H^{-,N}),\psi_p^H\rangle_H-\langle\lambda_H^{+,N},\psi_p^H\rangle_H&=&0,\quad -N\leq p\leq N,\\
\langle \lambda_H^{-,N},\psi^H_q\rangle_H-\langle F_H(\lambda_H^{+,N}),\psi_q^h\rangle_H&=&0,\quad -N\leq q\leq N,\\
\langle \lambda_0^{-,N},\psi_r^0\rangle_0-\langle F_0(\lambda_0^{+,N}),\psi_r^0\rangle_0&=&\langle f_-,\psi_r^0\rangle_0,\quad -N\leq r\leq N,\\
\langle\lambda_0^-,\psi_s^0\rangle_0-2\frac{s}{c}\eta \langle \hw(\lambda_0^-,\lambda_H^-),\psi_s^0\rangle_0-\langle\lambda_0^+,\psi_s^0\rangle_0&=&\langle f_+,\psi_s^0\rangle_0,\quad -N\leq s\leq N.
\end{eqnarray*}
In our calculations we use a finite element approximation to $w$ based on standard quadrilateral elements and continuous mapped piecewise bilinear functions. The above system is solved using deal.II~\cite{dealII82}.

\red{In the upcoming calculations we can choose $N$ to be relatively small (in fact $N=10$ in the first experiment) because we expect only a few propagating modes in the frequency domain.  Then $N$ needs to be chosen to include these modes and a few evanescent modes in addition.}

\section{Numerical Results}\label{numer}
\subsection{Quantification of the error}
We start with a simple problem with a known exact solution to test convergence. We fix $c$ and suppose $\epsilon=\epsilon_-\delta(t)$ for $z<h_i$ and $\epsilon=\epsilon_+\delta(t)$ for $z>h_i$ for some fixed interface height $h_i$ with $0<h_i<H$.  Consider an incident field
at normal incidence (so $d_1=0$) defined by
\[
u^i(x,z,t)=f\left(-\frac{\sqrt{\epsilon_+}}{c}(z-h_i)+t\right),\quad z<h_i,
\]
where $f$ is a given function.
Then the scattered field for  $z<h_i$ will be
\[
u^s(x,z,t)=g_s\left(\frac{\sqrt{\epsilon_+}}{c}(z-h_i)+t\right)
\]
for some function $g_s$ and the transmitted wave in $z>h_i$ is
\[
u^t(x,z,t)=g_t\left(-\frac{\sqrt{\epsilon_-}}{c}(z-h_i)+t\right)
\]
for some function $g_t$.
Continuity of the fields at $z=h_i$ implies
\begin{equation}
f(t)+g_s(t)=g_t(t)\mbox{ for all }t.\label{fgh}
\end{equation}
Continuity of the normal derivative implies
\[
\frac{\sqrt{\epsilon_+}}{c}f'(t)-\frac{\sqrt{\epsilon_+}}{c}g_s'(t)=\frac{\sqrt{\epsilon_-}}{c}g_t'(t)
\]
Hence, integrating this expression and using causality,
\[
-\sqrt{\epsilon_+}f(t)+\sqrt{\epsilon_+}g_s(t)=-\sqrt{\epsilon_-}g_t(t).
\]
Using (\ref{fgh})
\[
g_s(t)=\frac{\sqrt{\epsilon_+}-\sqrt{\epsilon_-}}{\sqrt{\epsilon_+}+\sqrt{\epsilon_-}}f(t)
\]
and
 \[
g_t(t)=f(t)+g_s(t)=\frac{2\sqrt{\epsilon_+}}{\sqrt{\epsilon_+}+\sqrt{\epsilon_-}}f(t).
\]
We choose $f(t)$ to be sufficiently smooth for our convergence theorem to hold and to vanish for $t<0$, and in particular for parameters $m$,  $\alpha_{\rm{}inc}>0$ and $\beta_{\rm{}inc}$ we define
\begin{equation}
f(t)=\left\{\begin{array}{ll} 0 & \mbox{ for } t<\beta_{\rm{}inc}\mbox{ or }t>\pi/\alpha_{\rm{}inc}+\beta_{\rm{}inc}\\
{\rm{}sin}^{m}(\alpha_{\rm{}inc} (t-\beta_{\rm{}inc}))&\mbox{ for } \beta_{\rm{}inc}<t<\pi/\alpha_{\rm{}inc}+\beta_{\rm{}inc}
\end{array}
\right.
\label{incident}
\end{equation}
This function is in $H^m(\mathbb{R})$, and we choose $m=4$.  We  choose $\alpha_{\rm{}inc}=4$ and $\beta_{\rm{}inc}=0.5$ in this example.

We can now solve the above problem with $c=1$, $\epsilon_-=1$ and $\epsilon_+=4$ to generate a solution to the scattering problem and hence test convergence of the time stepping method using the fixed spatial mesh in Fig.~\ref{meshtest} and BDF2 in time. The final time is $T=4$ by which time the wave has essentially exited the computational region. Results are shown in Fig.~\ref{error_plots}.  Both in the $L^2$ and $H^1$ norm the convergence rate is ultimately $O((\Delta t)^2)$ as expected from
Theorem~\ref{tm_conv}.  No instability is evident.
\begin{figure}[t]
\begin{center}
\resizebox{0.5\textwidth}{!}{\includegraphics{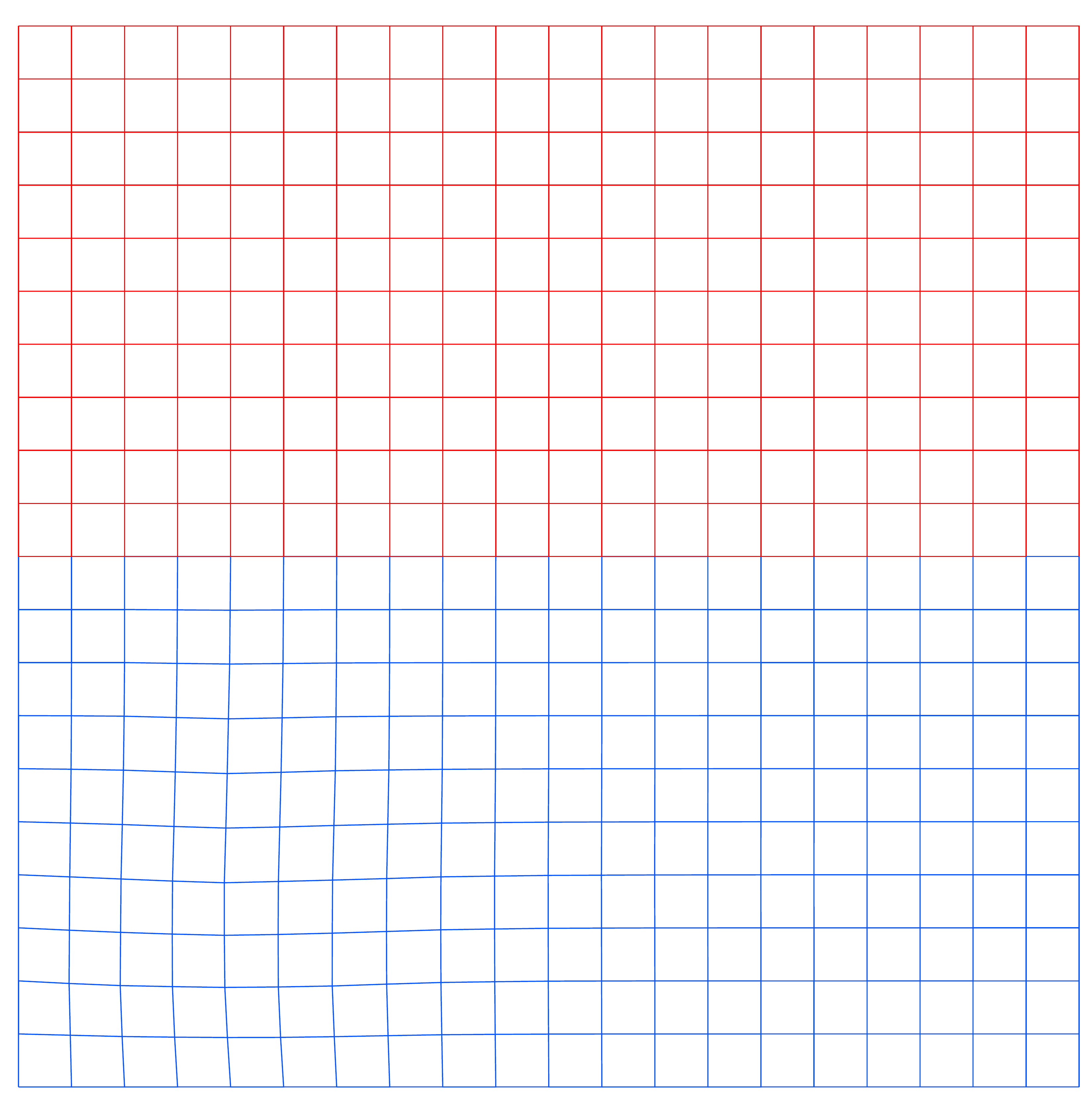}}
\end{center}
\caption{Simple almost uniform spatial mesh used to investigate convergence of the time stepping scheme.  The domain is $\Omega=[0,1]\times[0,1]$, and the interface is at $h_i=1/2$.}
\label{meshtest}
\end{figure}

\begin{figure}[h]
\begin{center}
\begin{tabular}{cc}
\resizebox{0.45\textwidth}{!}{\includegraphics{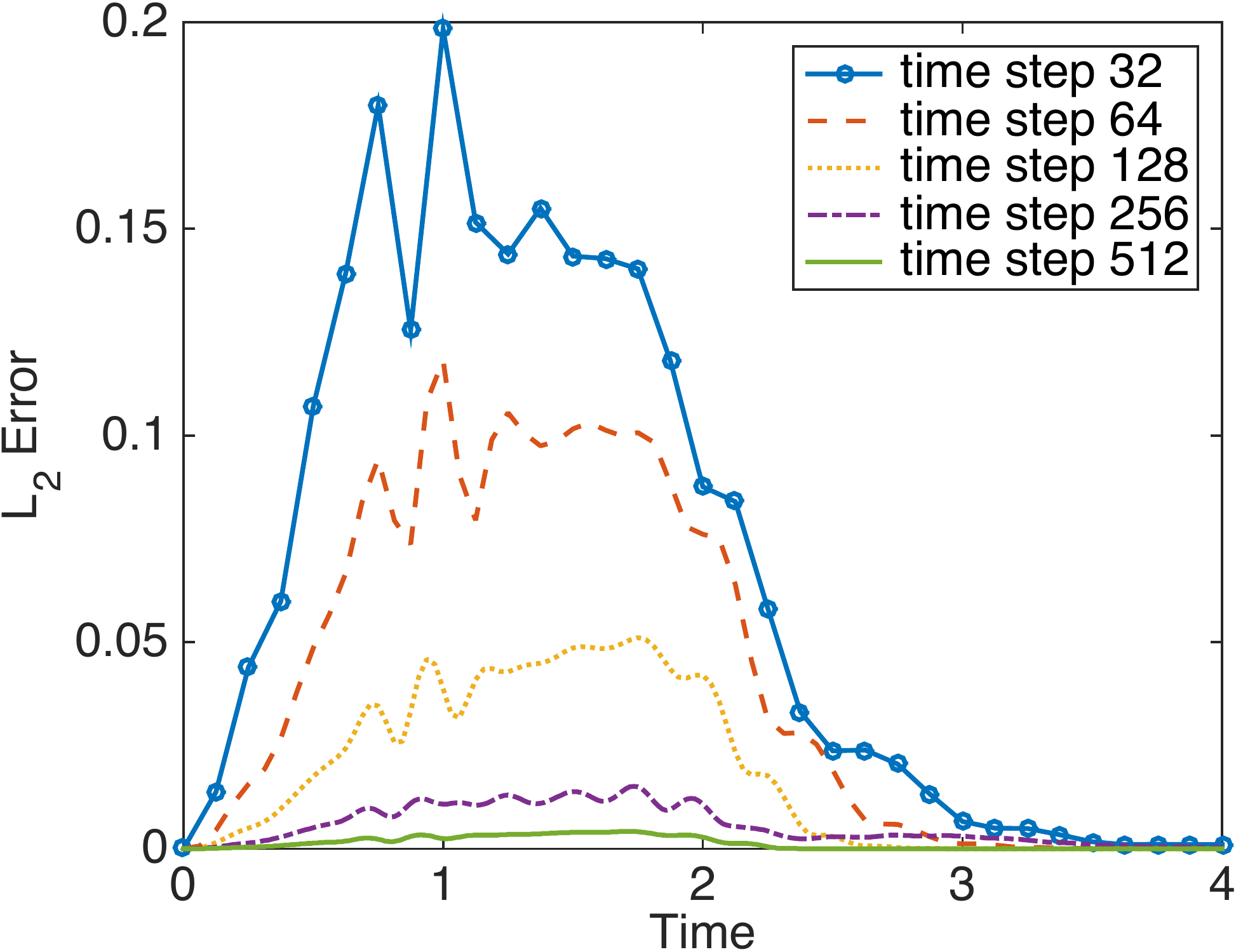}}&
\resizebox{0.45\textwidth}{!}{\includegraphics{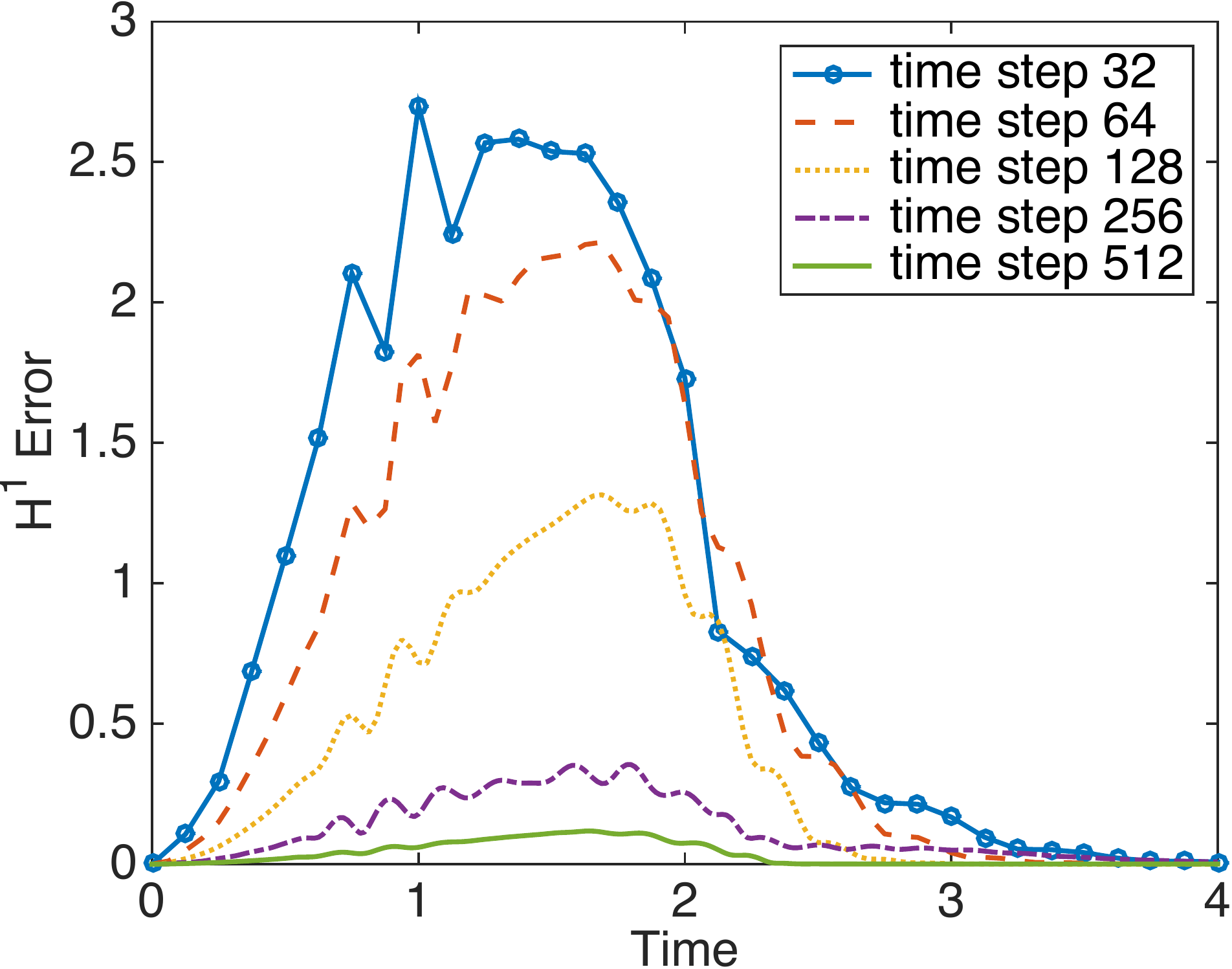}}\\
\resizebox{0.45\textwidth}{!}{\includegraphics{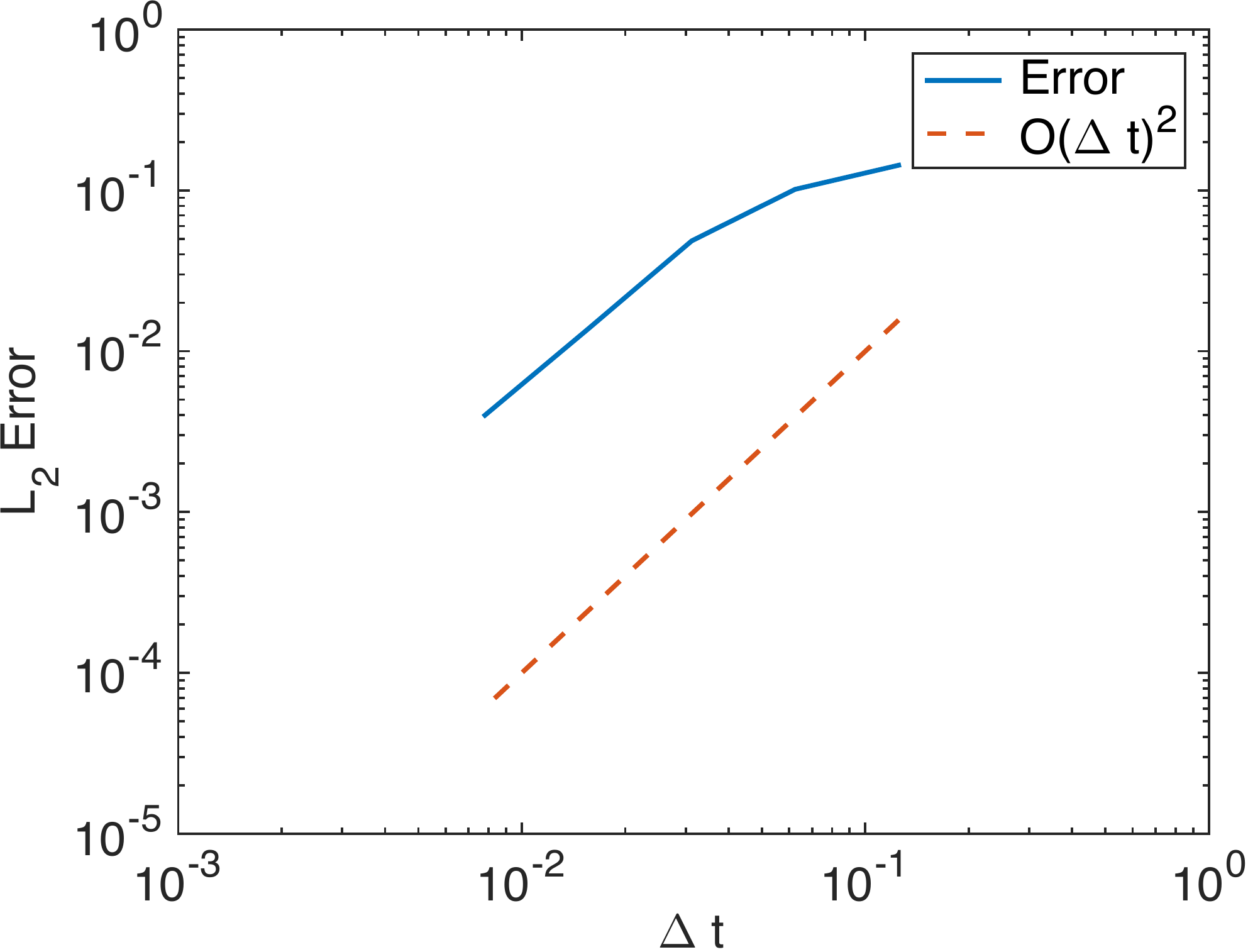}}&
\resizebox{0.45\textwidth}{!}{\includegraphics{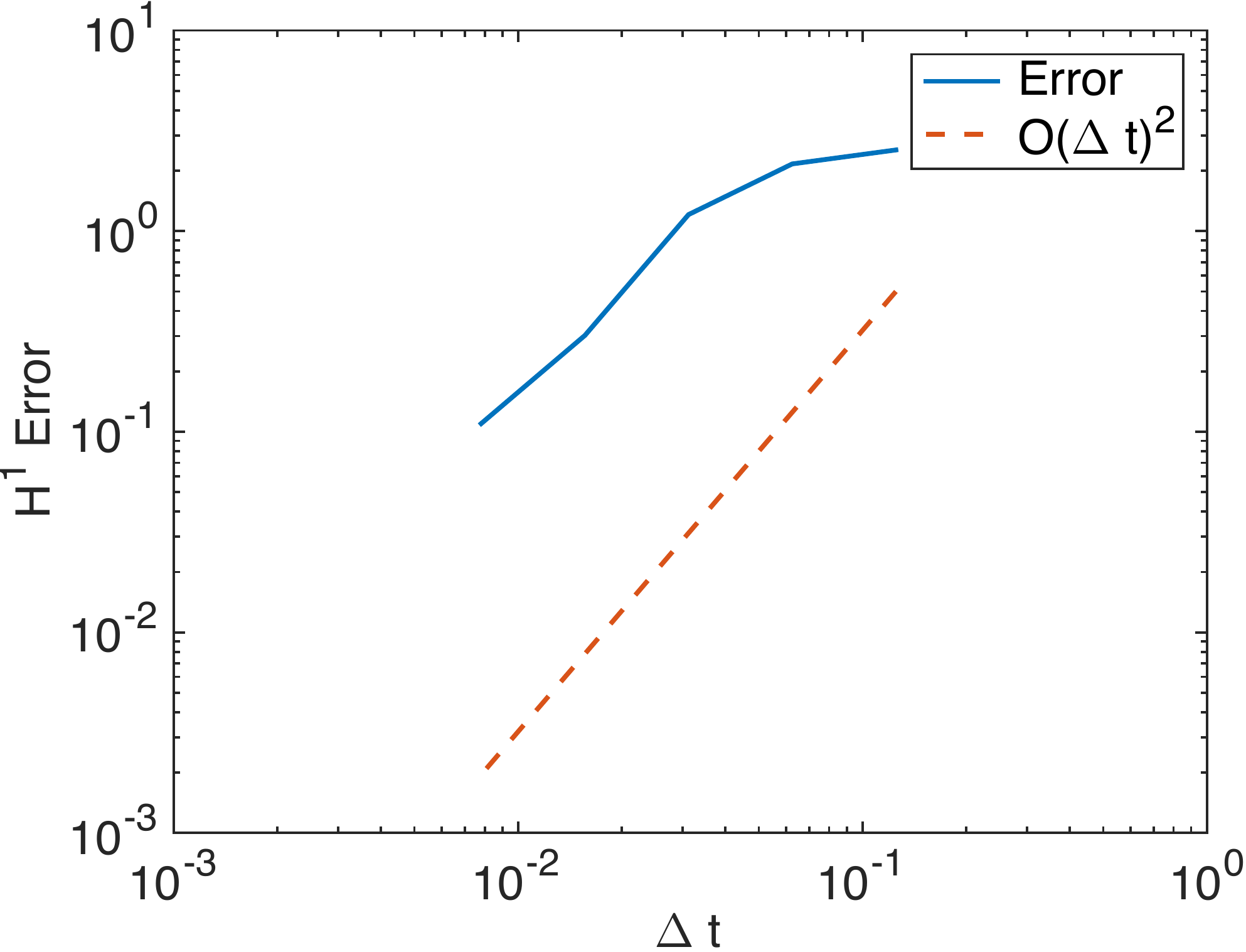}}
\end{tabular}
\end{center}
\caption{Time step convergence of the method.  Top row: The \red{absolute} $L_2$ (left) and $H^1$ norm (right) error at different times.  Bottom row: \red{to check the convergence rate} we show a log-log plot of the \red{absolute} $L_2$ (left) and $H^1$ error (right) as a function of $\Delta t$ at $t=1.5$.  We also show a $O((\Delta t)^2)$ reference line.}
\label{error_plots}
\end{figure}

 \subsection{Frequency dependent materials}\label{freq_dep}
 Our Assumption~\ref{A1}  on the coefficients in  the differential equation handles at least two important cases relevant to  practical applications to solar-voltaic components.  To see why this is necessary note that in thin film devices the size of components is close to the wavelength of light.  At these frequencies (for example corresponding to a free space wavelength of 500nm) metals can no-longer be modeled as perfect conductors and the light penetrates an appreciable distance into the metal.  For example, at a free space wavelength of 500nm, $\hat{\epsilon}_r=-2.4683 + 3.1173i$ for gold. So in fact $\hat{\epsilon}_r$ is often complex valued and
 the real part may not be positive.  In a Drude model (commonly used to model
 metals \cite[Ch. 2]{ashcroft_book}) we have
 \begin{equation}
\hat{b}_{\rm{}m}(s)=\alpha_m+\frac{\beta_m}{s(1+\gamma_m s)}\label{drude}
 \end{equation}
 for positive, perhaps spatially dependent, real constants $\alpha_m$, $\beta_m$ and $\gamma_m$.  Note that if $\gamma_m=0$ this reduces to the usual model of conductivity.
 To verify property 2 of Assumption~\ref{A1} in the case of a Drude model note that
 \begin{eqnarray*}
\Re( s(\hat{b}_{\rm{}m}(s)-d_1^2))&=&\Re\left( s\alpha_m +\frac{(1+\gamma_m\overline{s})}{|1+\gamma_ms|^2}-d_1^2s
\right)\\
&=&\sigma\left(\alpha_m-d_1^2+\frac{|s|^2}{\sigma|1+\gamma_m s|^2}(1+\sigma\gamma_m)\right)\\
&\geq& \sigma\left(\alpha_m-d_1^2\right).
 \end{eqnarray*}
 Provided, for example, $\alpha_m-d_1^2>\gamma_0>0$ we have the desired positivity.
 
Because of the wide range of frequencies need to simulate a solar cell components across the solar spectrum, it is also necessary to take into account frequency dependence for dielectics. Dielectric components are often modeled as having no absorption but frequency dependence (so $\Im(\hat{b}(s))=0$). A commonly used model
is the Sellmeier equations~\cite[page 472]{jenkins_book}. In this case, the simplest model is
\begin{equation}
\hat{b}_{\rm{}s}(s)=1+\frac{\alpha_s}{1+\beta_s s^2}.
\label{sell}
\end{equation}
More generally there are usually sums of rational functions of the same form as above. Here $\alpha_s$ and $\beta_s$ are positive real constants.  This model also fits into the theory because
\begin{eqnarray*}
\Re( s(\hat{b}_{\rm{}s}(s)-d_1^2))&=&\Re\left( s+\alpha_s\frac{(s+\beta_s |s|^2\overline{s})}{|1+\beta_s s^2|^2}-sd_1^2\right)\\
&=&\sigma\left(1+\alpha_s\frac{(1+\beta_s |s|^2)}{|1+\beta_s s^2|^2}-d_1^2\right)\\
&\geq& \sigma(1-d_1^2).
\end{eqnarray*}
So provided $1-d_1^2>\gamma_0>0$ we again have the necessary lower bound.

Neither of these models satisfy the Kramers-Kronig relationship that guarantees stability and causality, but, as we have proved, they still provide a stable time dependent response for finite time. 
Note that the Cauchy model of a dielectric \cite[page 468]{jenkins_book} does not fit into this theory.  

\begin{figure}[t]
\begin {center}
\begin{tabular}{ccc}
\resizebox{0.45\textwidth}{!}{\includegraphics{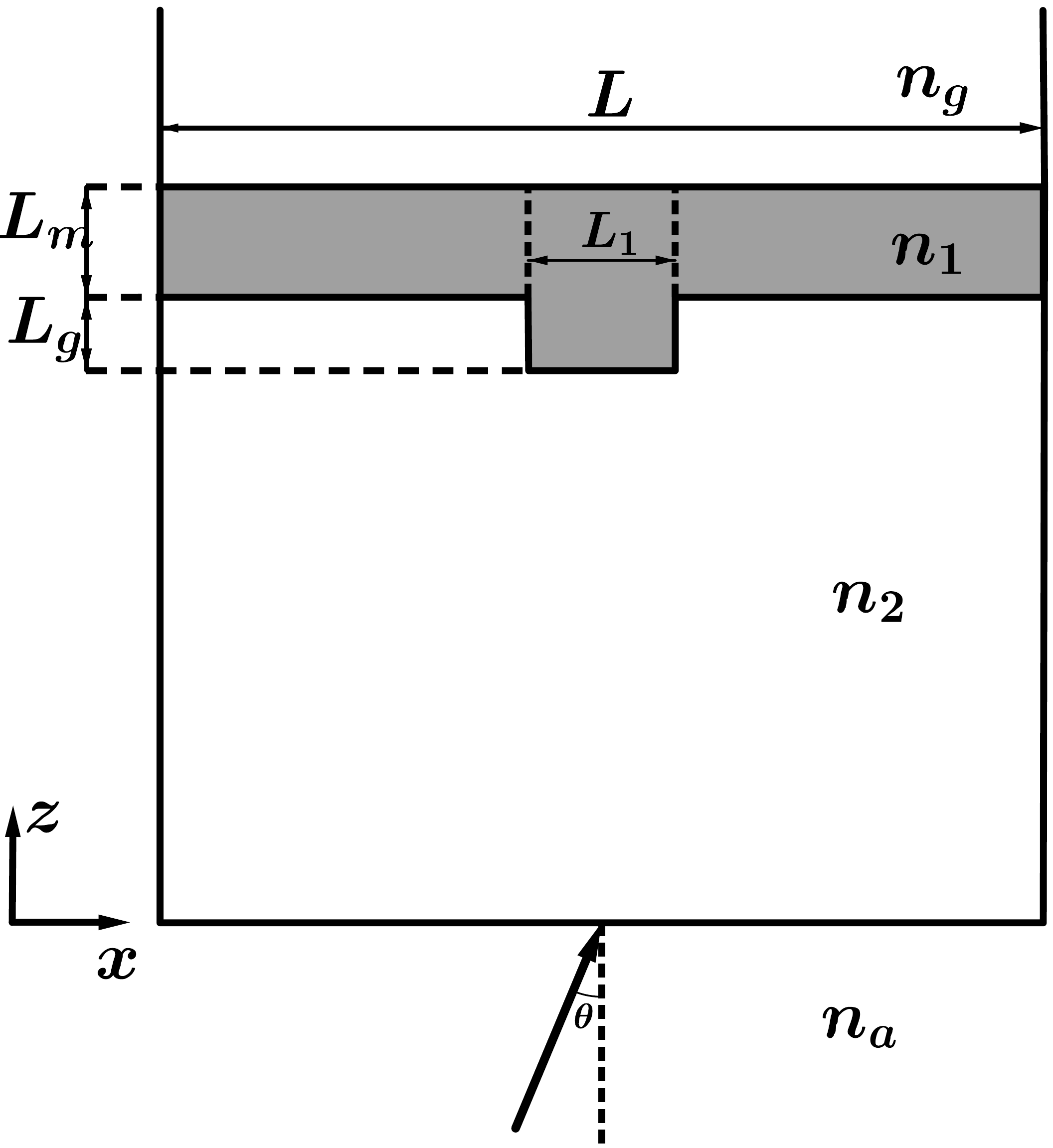}}&\hspace*{.2in}
\resizebox{0.45\textwidth}{!}{\includegraphics{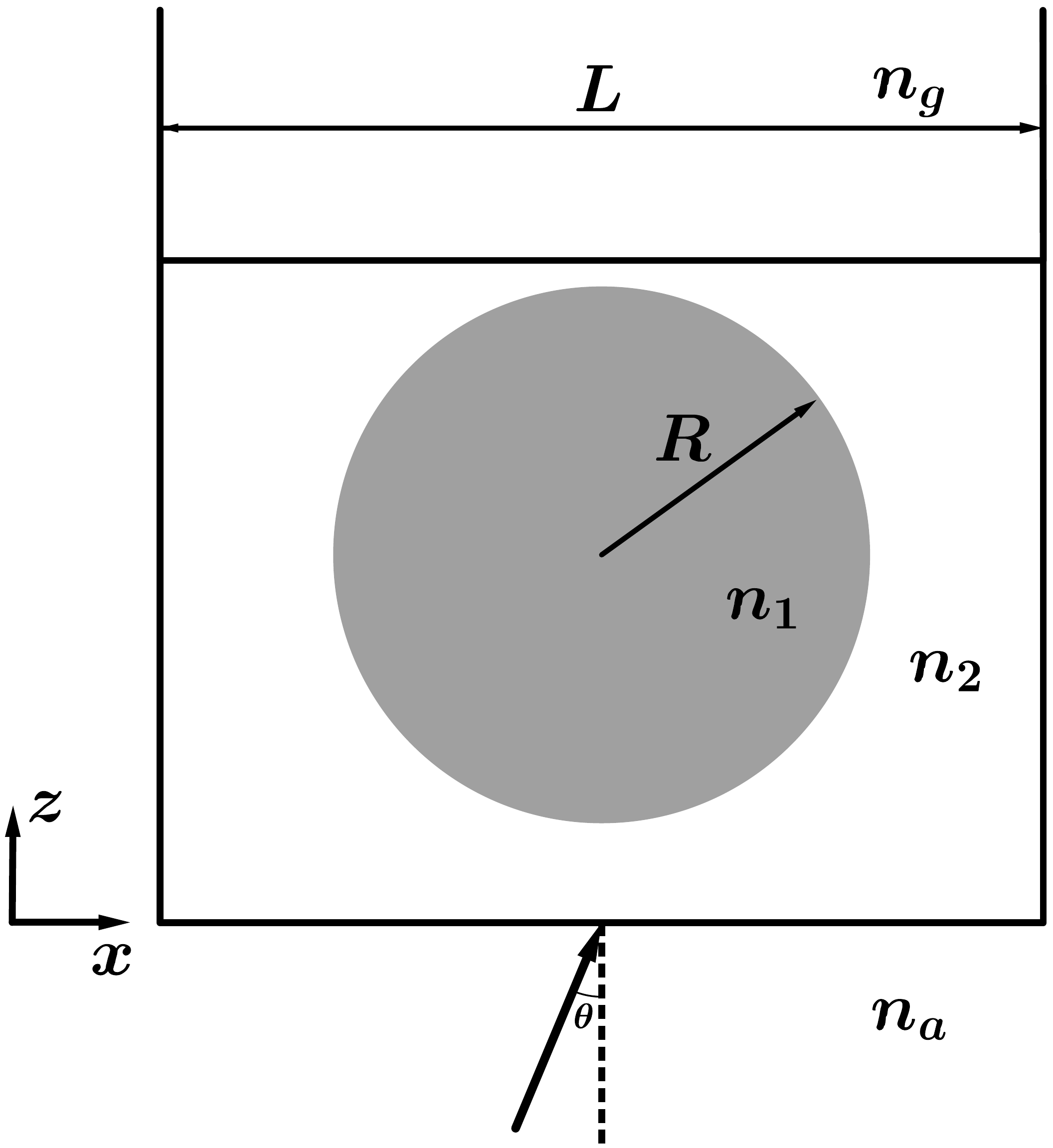}}
\end{tabular}

\end{center}
\caption{Two gratings used for the numerical results on frequency dependent media. In this figure $n_1$, $n_2$, $n_a$ and $n_g$ label the refractive indices for the various subdomains.  The permittivity is obtained by squaring the refractive index so $\epsilon_r=n_1^2$ in the grating in the left example. The
left figure shows a simple metallic grating in air.  The right figure shows a cylinder slightly above a glass
substrate and features a frequency dependent refractive index in the infinite glass region.}
\label{cartoon1}
\end{figure}

\subsubsection{Scattering from a Drude metal}
We now show an example of a typical grating (without other components of a thin film solar cell in order to emphasize grating effects).  The geometry of the experiment is shown in the left panel of Fig.~\ref{cartoon1}.  Light is incident (at incidence  $\theta= 6^\circ$) on the thin metal grating which is modeled by a fictional Drude medium (see (\ref{drude})) having parameters
 $\alpha_m = 4.0$, $\beta_m=10.0$ and $\gamma_m=0.5$.   We set $n_2$ and $n_g$ 
to be 1 and chose 10 modes for the top and bottom boundaries. By setting $m=4$, $\alpha_{\rm{}inc }=4.0$ and
$\beta_{\rm{}inc}= 0$, we chose (\ref{incident}) as the incident function for this example.  In particular the incident field has a non zero Fourier transform at low frequencies where our fictitious Drude model has negative real part, so that the dispersive and dissipative nature of the medium is probed. 
\begin{figure}[h]
\begin {center}
\begin{tabular}{cc}
\resizebox{0.45\textwidth}{!}{\includegraphics{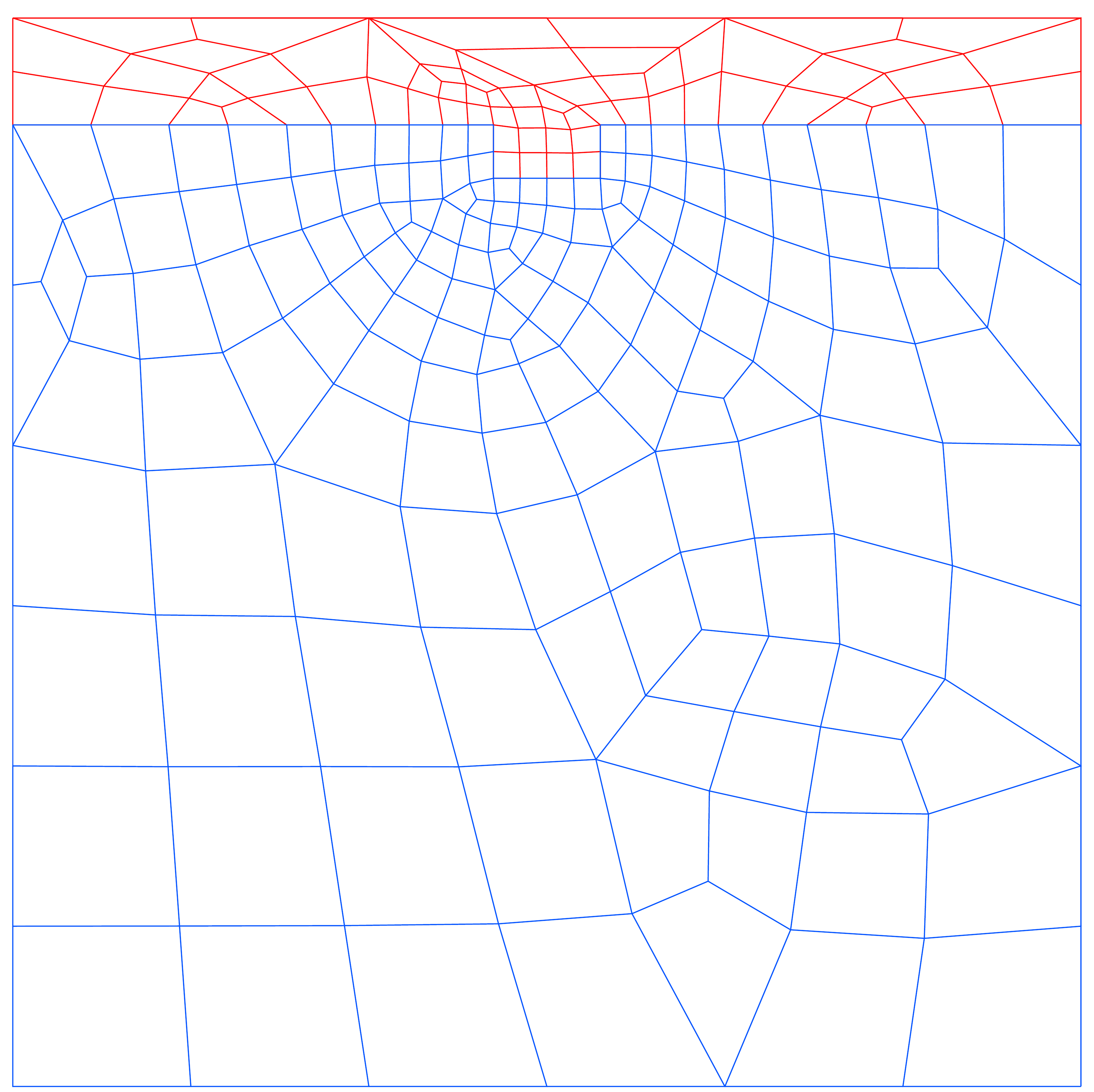}}&
\resizebox{0.5\textwidth}{!}{\includegraphics{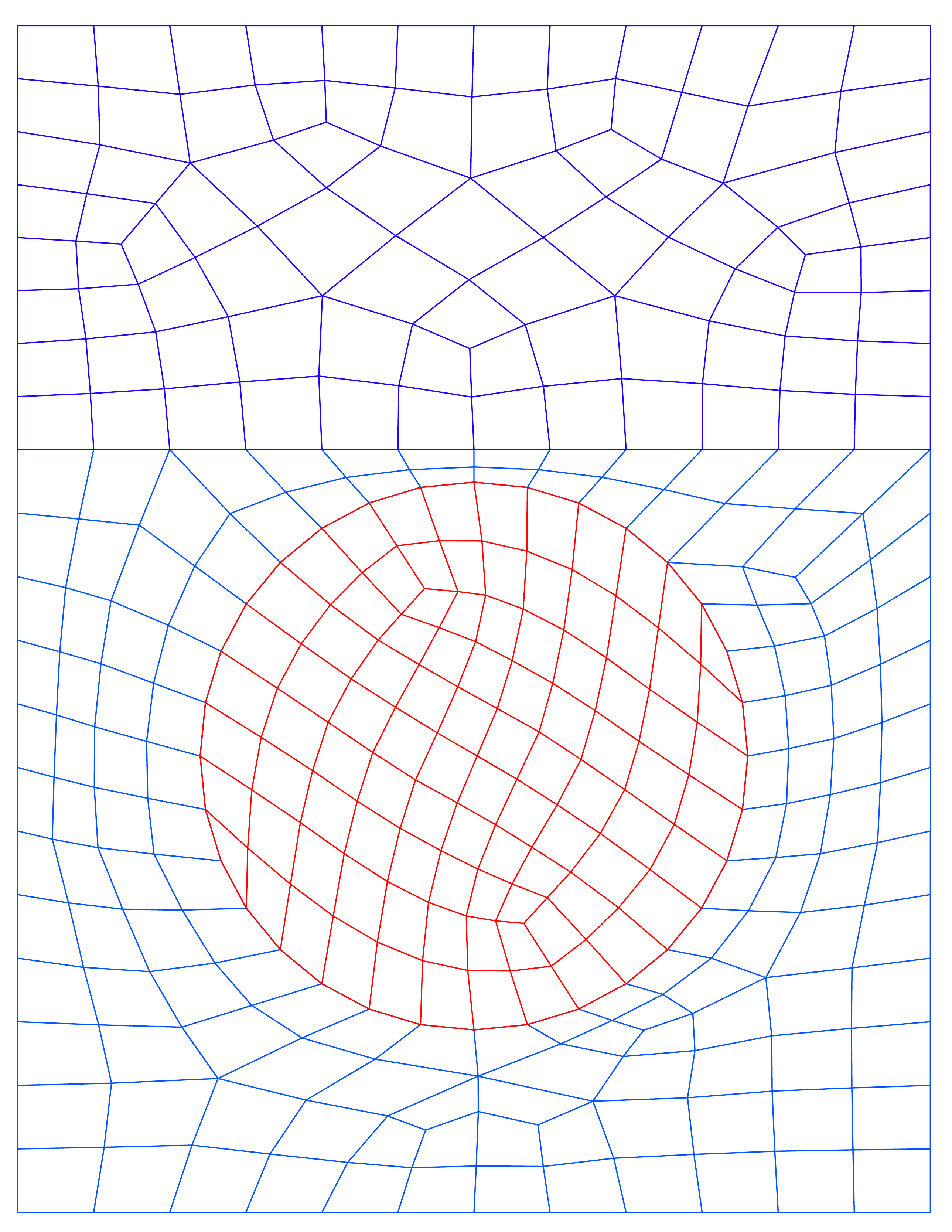}}
\end{tabular}
\end{center}
\caption{Mesh for the Drude model (left) and Sellmeier model (right).}
\label{drudesell}
\end{figure}

The domain is $\Omega=[0,1]\times[0,1]$ and so, using the notation in Fig.~\ref{cartoon1} left panel, $L =1$.
In addition we choose  $L_1 = 0.1$,  $L_m = 0.1$ and $L_g = 0.05$. The speed of light is set to $c=1$ and we integrate until $T=4$ using 512 time steps so $\Delta t=0.0078$. The mesh is from Fig.~\ref{drudesell} left panel.  This spatial mesh was generated using {\tt{}gmsh}~\cite{gmsh} a triangular mesh generator that can post-process the mesh to create a quadrilateral mesh.  Obviously the resulting mesh is rather poor but this is useful to test the sensitivity of the method to mesh perturbations. Density plots of the computed total field are shown in Fig.~\ref{drudesnap}.  At early times the incident field is clearly visible followed by a strong scattered field (in a thin film solar cell, other structures would serve to trap the energy).  In addition the transmitted field into the metal can be seen.  At late times a component of the field running along the metallic boundary is also visible and these may be related to surface plasmon polaritons \cite{AtPo10}.  No exact solution exists and this example is intended to show that the 
time stepping scheme can indeed handle a Drude model in a stable way.
\begin{figure}[htb]
\begin{center}
\begin{tabular}{cc}
\resizebox{0.5\textwidth}{!}{\includegraphics{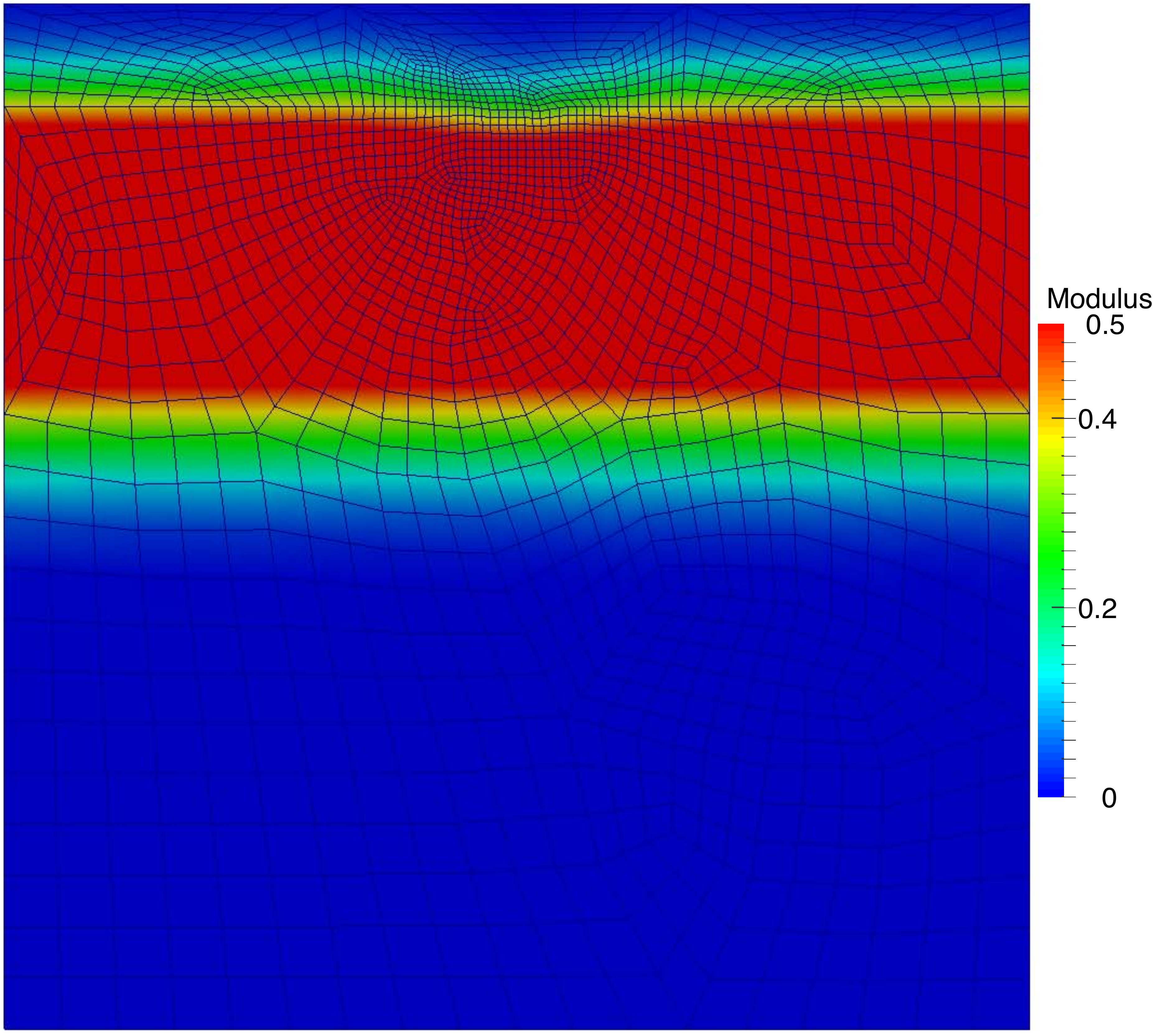}}&
\resizebox{0.5\textwidth}{!}{\includegraphics{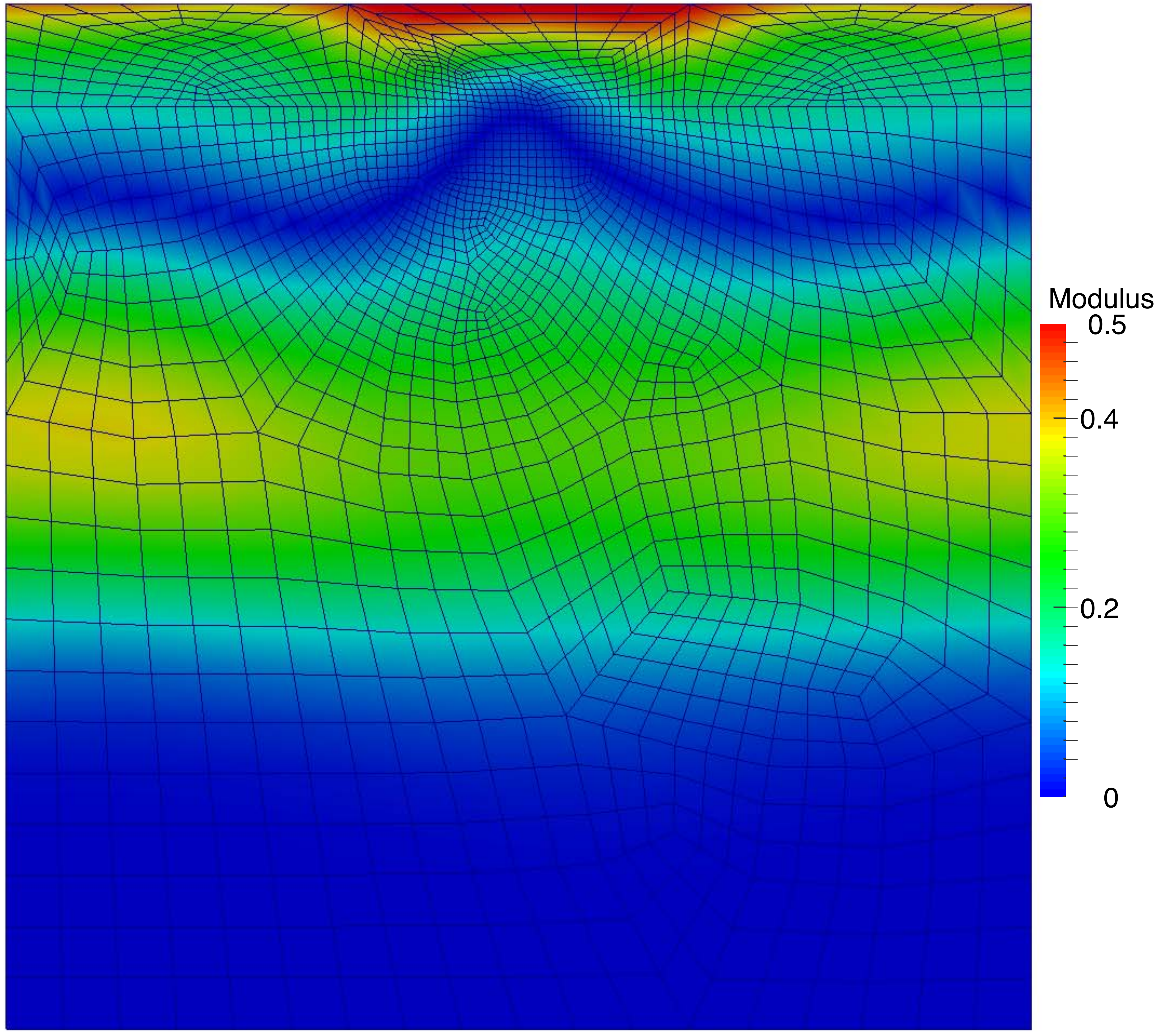}}\\
\resizebox{0.5\textwidth}{!}{\includegraphics{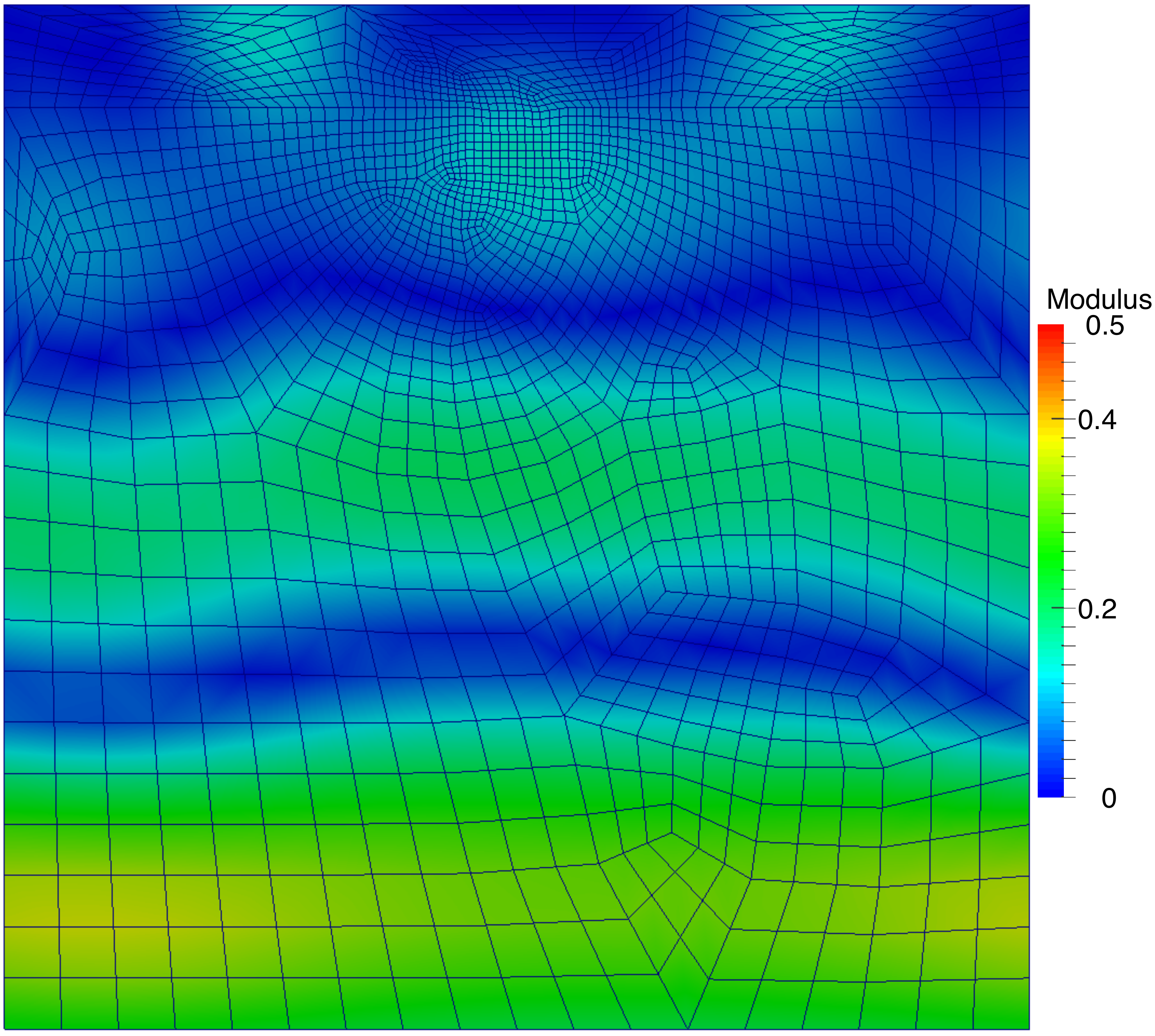}}&
\resizebox{0.5\textwidth}{!}{\includegraphics{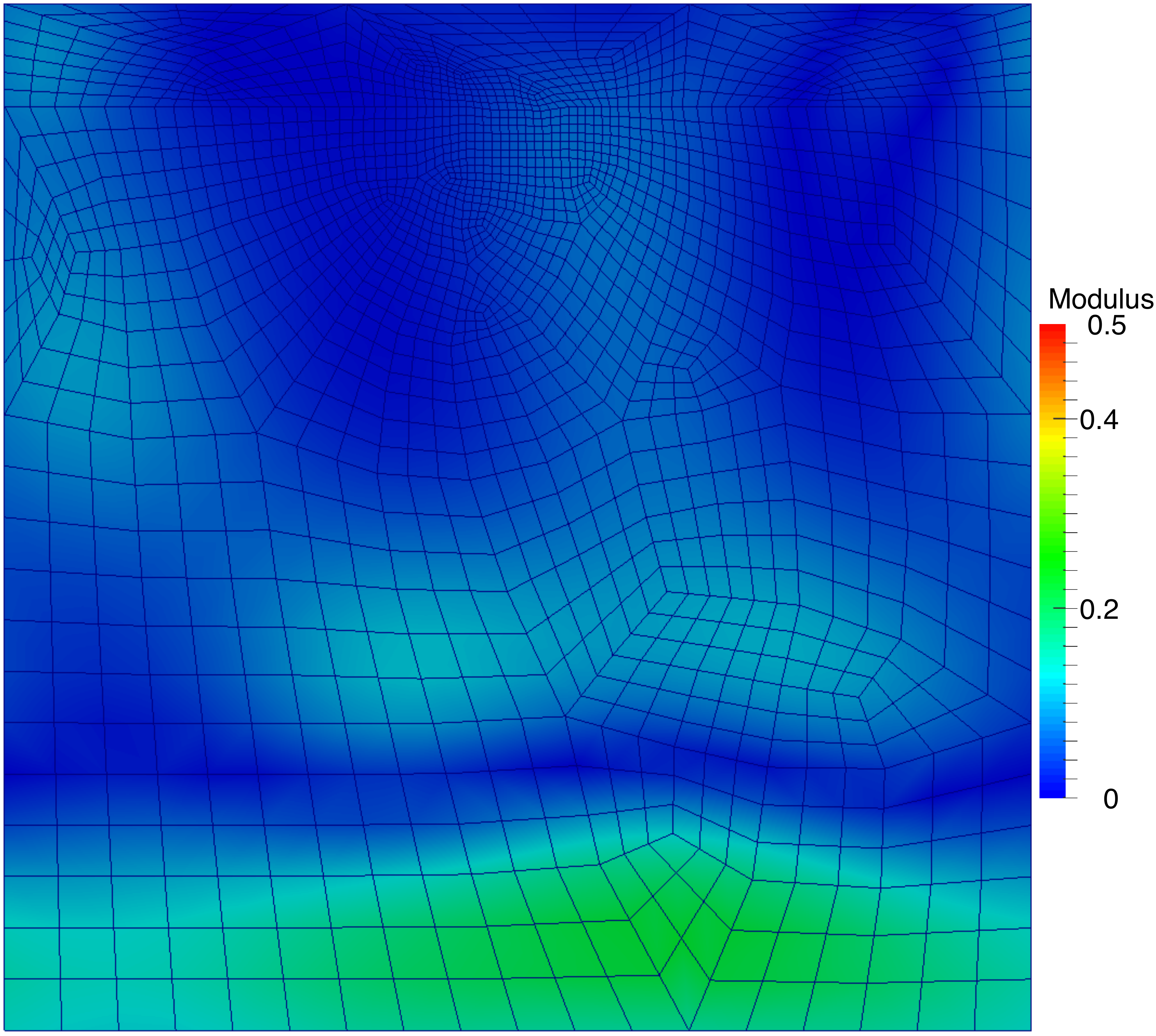}}\\
\resizebox{0.5\textwidth}{!}{\includegraphics{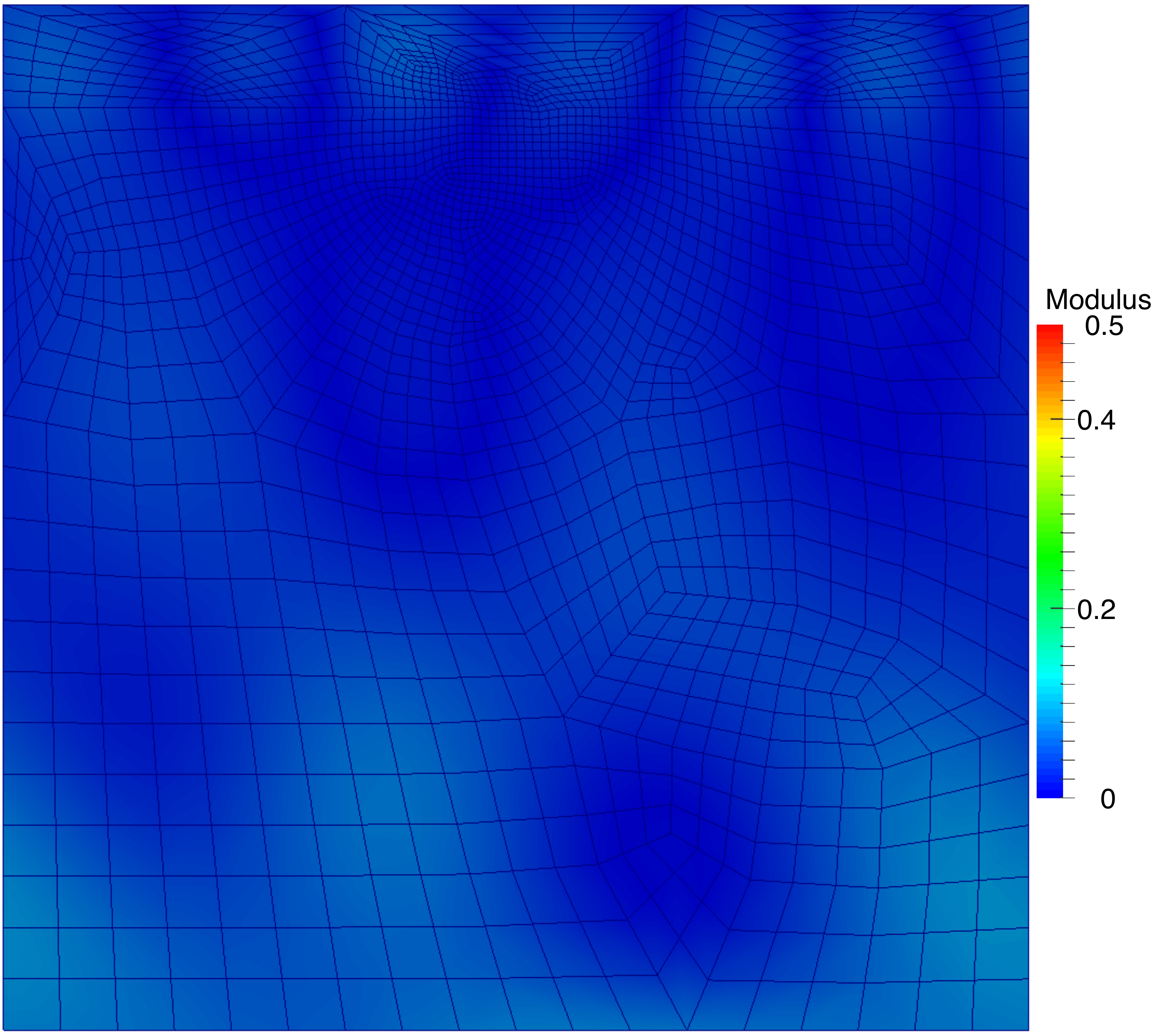}}&
\resizebox{0.5\textwidth}{!}{\includegraphics{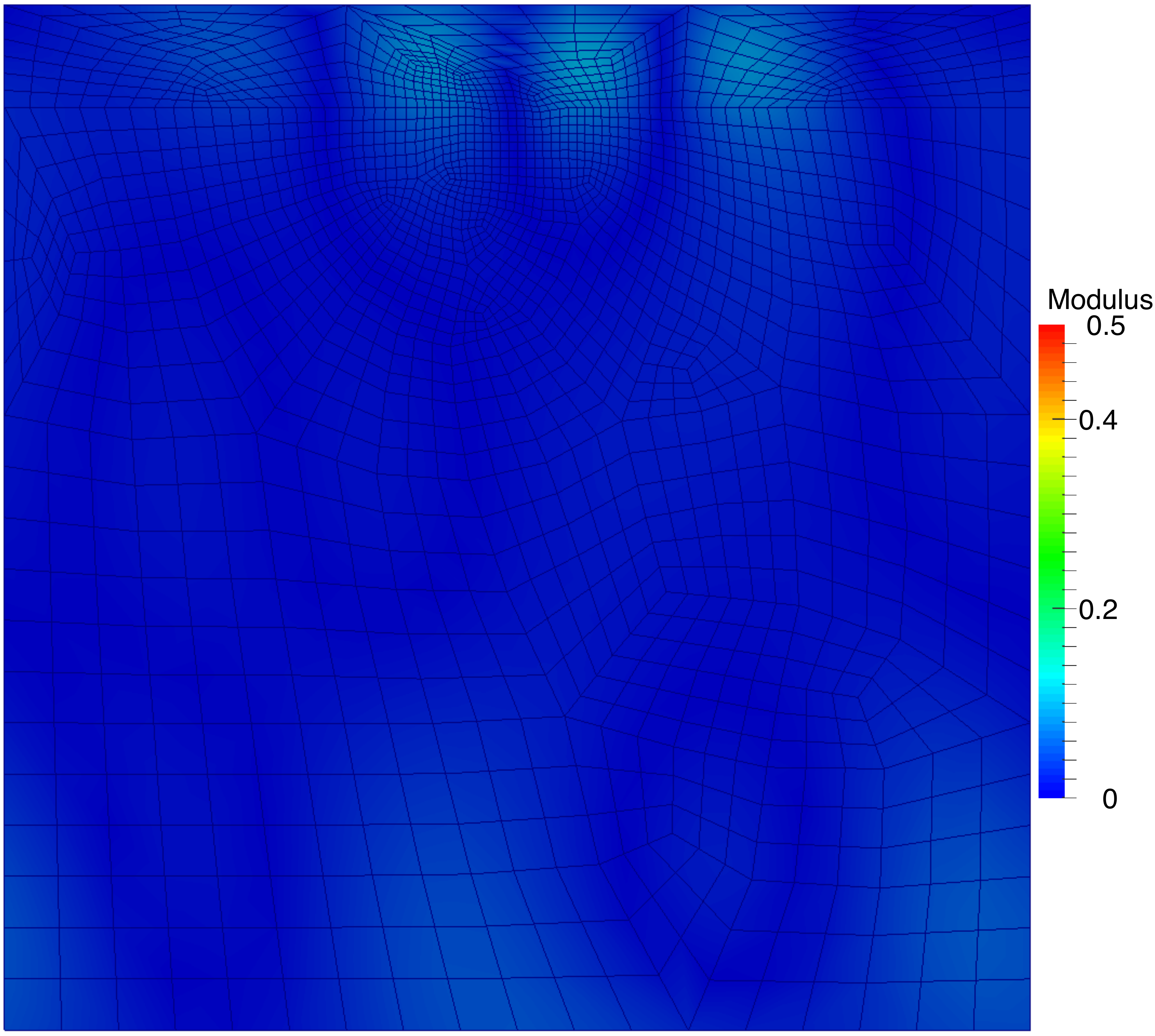}}
\end{tabular}
\end{center}
\caption{Snapshots of density plots the absolute value of the  total field $w(\bfx,t)$ for the Drude model at times $t=1.27,1.72,2.17,2.62,3.07,3.52$ (from top left to bottom right). At later times faint waves traveling close to the surface of the metal are suggestive of surface plasmon polaritons \cite{AtPo10}.}
\label{drudesnap}
\end{figure}

\subsubsection{Scattering from a Sellmeier dielectric}
The next model is an example for the Sellmeier model of a dielectric (see (\ref{sell})) taken from
\cite{Li_15_1} where a circular scatterer of constant permittivity sits on a glass substrate. The dimension of the unit cell and other parameters
are chosen so that Fourier components of an incident wave with wavelength greater than 650\,nm are
scattered predominately in a different direction to those with a wavelength below 650\,nm and the device is called a spectrum splitter.

 The geometric configuration is shown in the right panel of Fig.~\ref{drudesell} 
where the parameters are as follows: $L = 560$\,nm and $R = 168$\,nm (so $R/L =0.3$).  In a slight departure from the optimized result in \cite{Li_15_1} we introduce a 20nm gap between the circular scatterer and the gas substrate (otherwise {\tt gmsh} could not generate a mesh).  The mesh used for the upcoming results is shown in the right hand panel of Fig.~\ref{drudesell}.

We set $n_1=1.8$, $n_2=1$, $c=0.3\,\mu$m/femtosecond and $T=8$\,femtoseconds \red{and using 512 time steps so $\Delta t = 0.0156$}.
We chose 10 modes for the top and bottom boundaries. The glass substrate is assumed to be
 SF11 glass~\cite{sf11} so the Sellmeier model is
 \[
\epsilon_r=n_g^2=1+\frac{1.73759695\lambda^2}{\lambda^2-0.013188707}+\frac{0.313747346\lambda^2}{\lambda^2-0.0623068142}+\frac{1.89878101\lambda^2}{\lambda^2-155.23629}
\]
with $\lambda$ denoting the free space wavelength in units of $\mu$meters.   More generally
the   Sellmeier model is
\[
\epsilon_r=1+\sum_{j=1}^3\frac{\alpha_j\lambda^2}{\lambda^2-\beta_j}
\]
Converting  first to the Fourier frequency domain and then to the  Laplace transform domain this becomes
\[
\epsilon_r=1+\sum_{j=1}^3\frac{\alpha_j}{1+\frac{\beta_j}{4\pi^2 c_0^2}s^2}
\]
As explained earlier, the above rational function easily fits into our theory.

We use the following function as the incident wave:
\begin{equation}
f(\tau) = {\rm{}sin}(2.899 \tau){\rm{}e}^{-2(\tau-3)^2}, \quad\tau\in\mathbb{R},\label{finc}
\end{equation}
and the angle of incidence is $\theta = 6^\circ$.  For a graph of this function see Fig.~\ref{graph}. Considering the Fourier transform of this function, it has a maximum Fourier coefficient at the switch point $650$\,nm and hence has Fourier components on either side of the switch.
\begin{figure}
\begin{center}
\resizebox{0.6\textwidth}{!}{\includegraphics{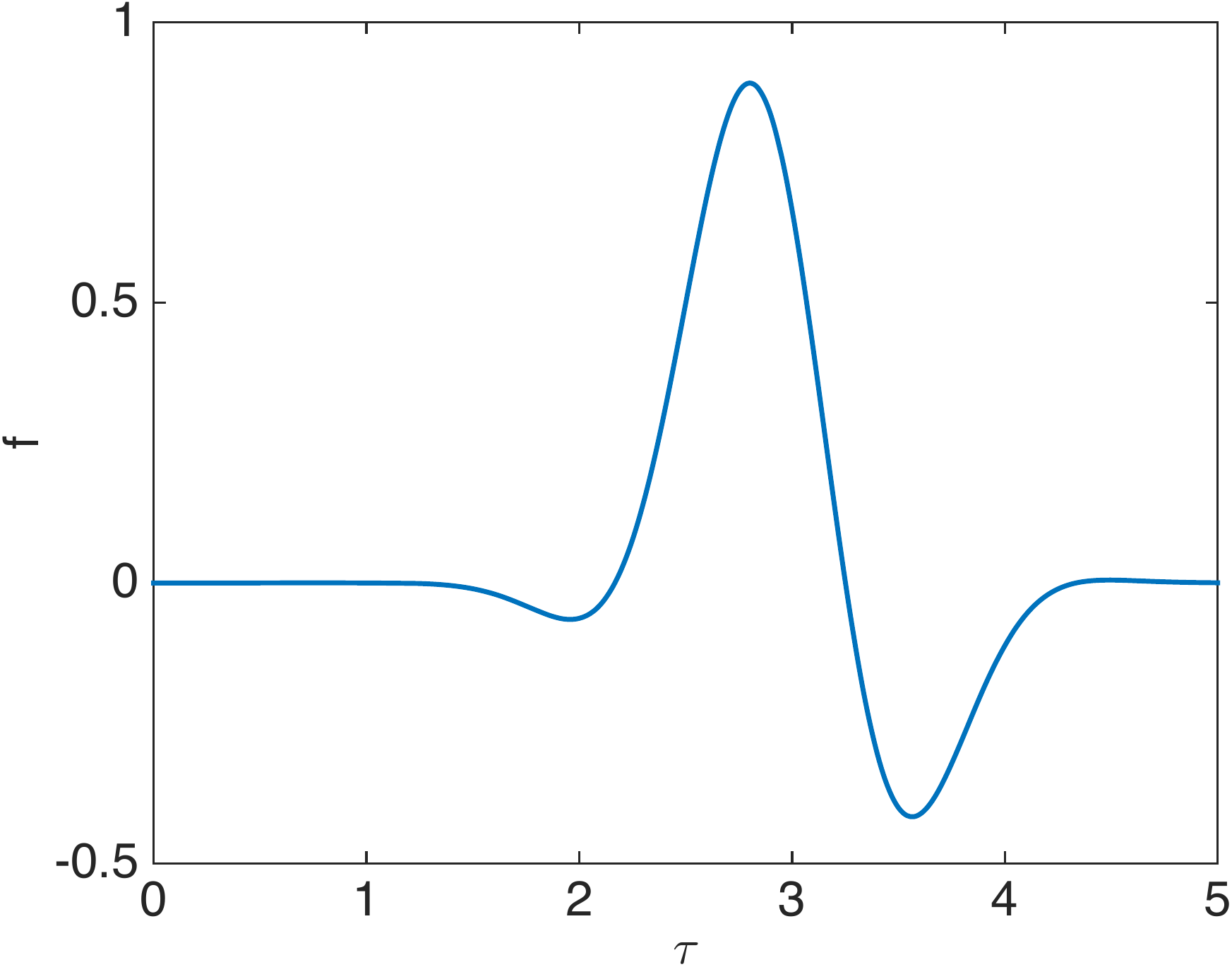}}
\end{center}
\caption{A graph of the function $f(\tau)$ against $\tau$ for the incident field given in equation (\ref{finc}).  The incident wave $w^i$ is given by
replacing $\tau$ by $t-Ld_1/c-d_2y/c$ as shown in (\ref{winc}).}
\label{graph}
\end{figure}

Results are shown in Fig.~\ref{sellsnap}.  Again we do not have an exact solution, but demonstrate that the solution is stable. Spectral splitting~\cite{Li_15_1} may be visible as two higher intensity zones moving in different directions (see  $t=5.29,6.14$ (a high intensity region moving up and to the right) and $t=6.99,7.84$ (a lower intensity region in the top left moving up slightly to the left)). 
\begin{figure}
\begin{center}
\begin{tabular}{cc}
\resizebox{0.42\textwidth}{!}{\includegraphics{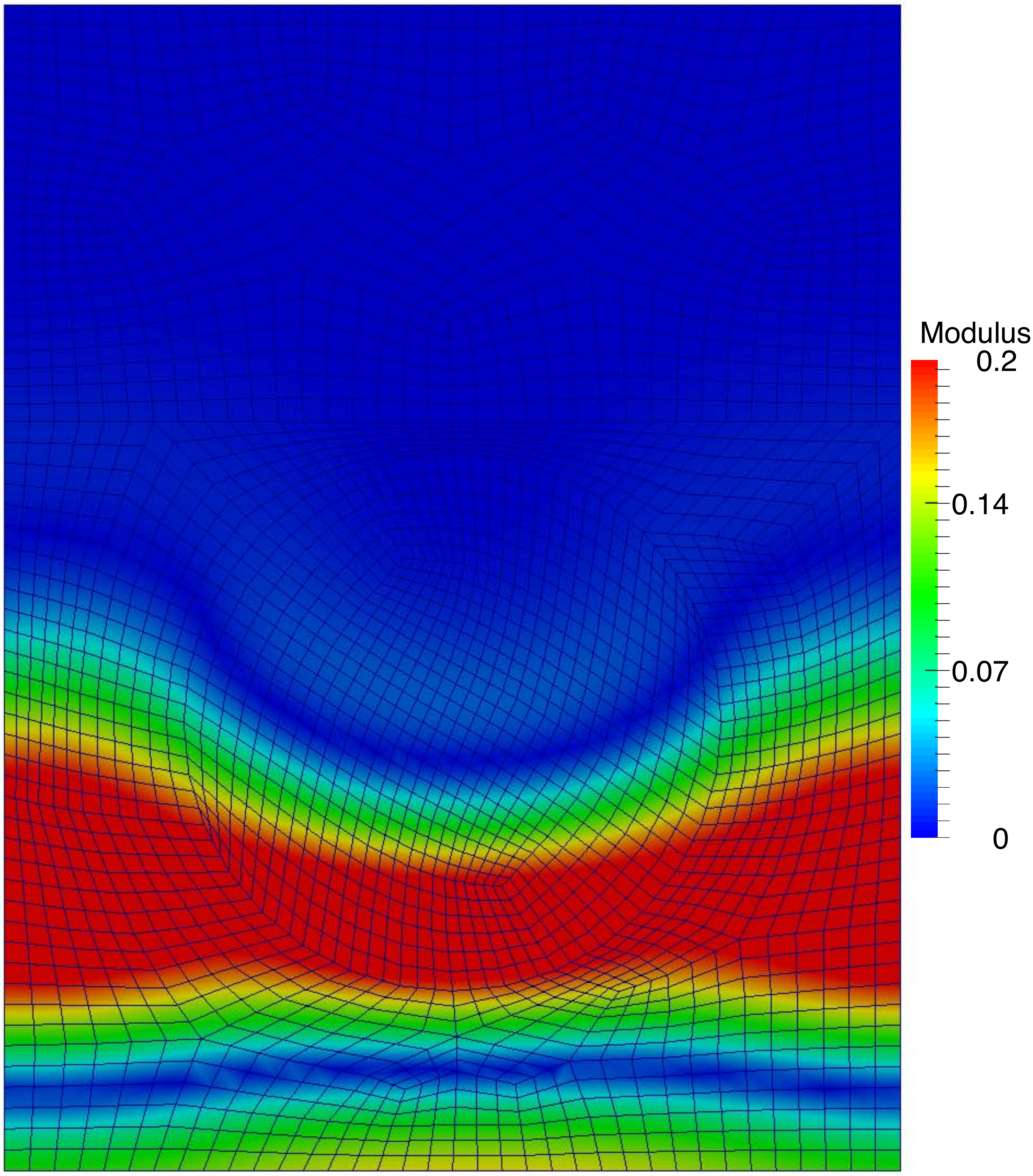}}&
\resizebox{0.42\textwidth}{!}{\includegraphics{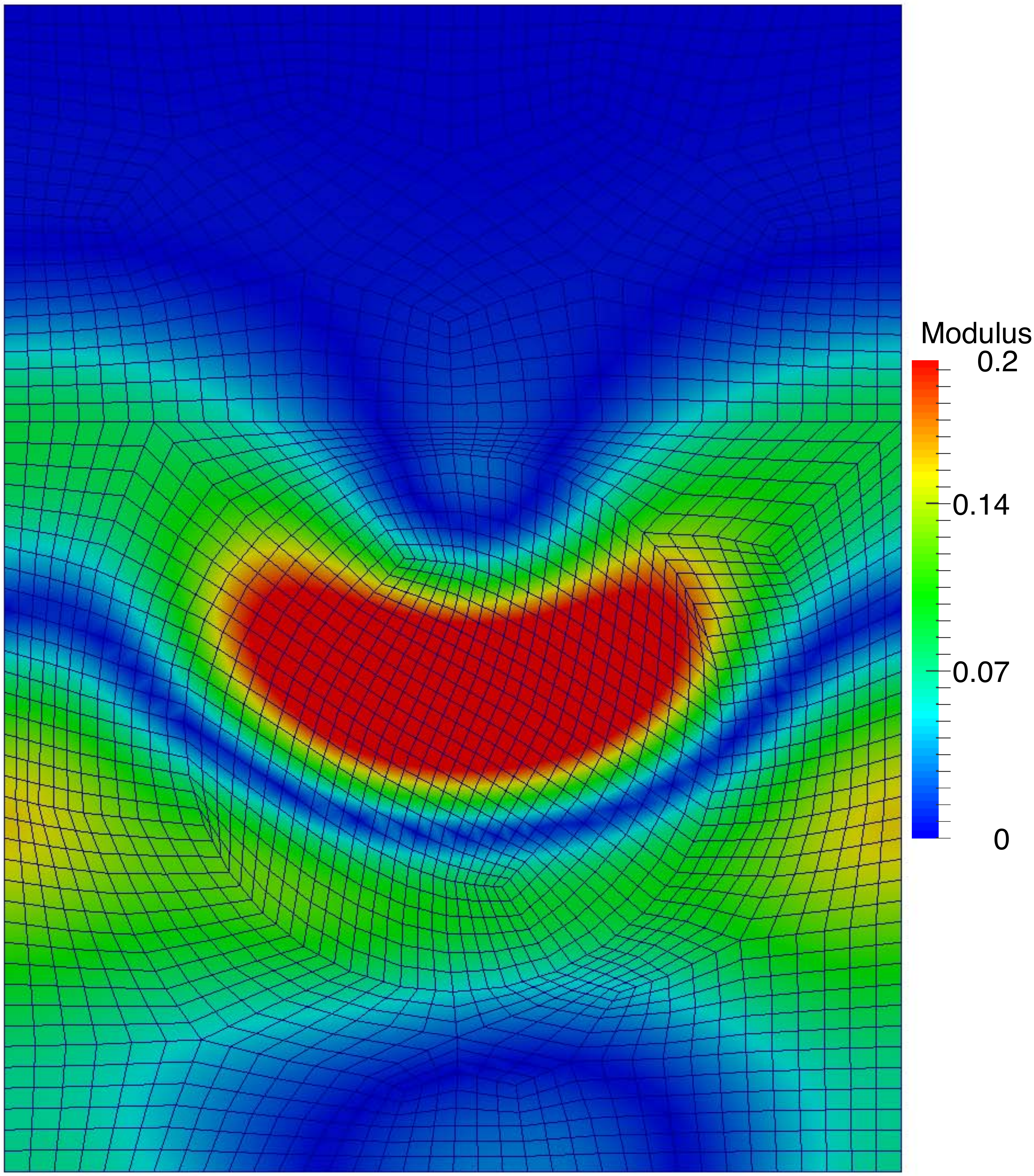}}\\
\resizebox{0.42\textwidth}{!}{\includegraphics{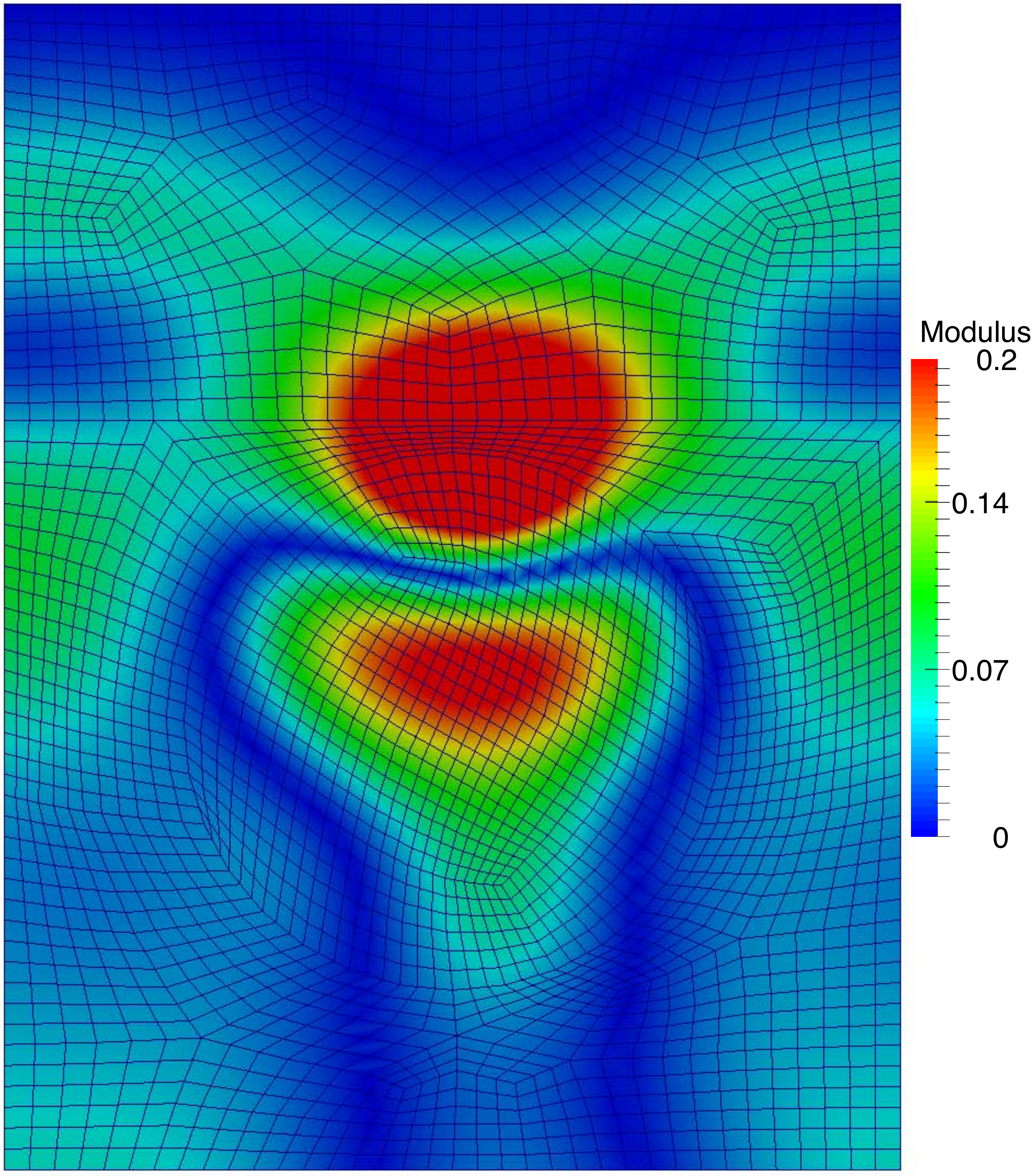}}&
\resizebox{0.42\textwidth}{!}{\includegraphics{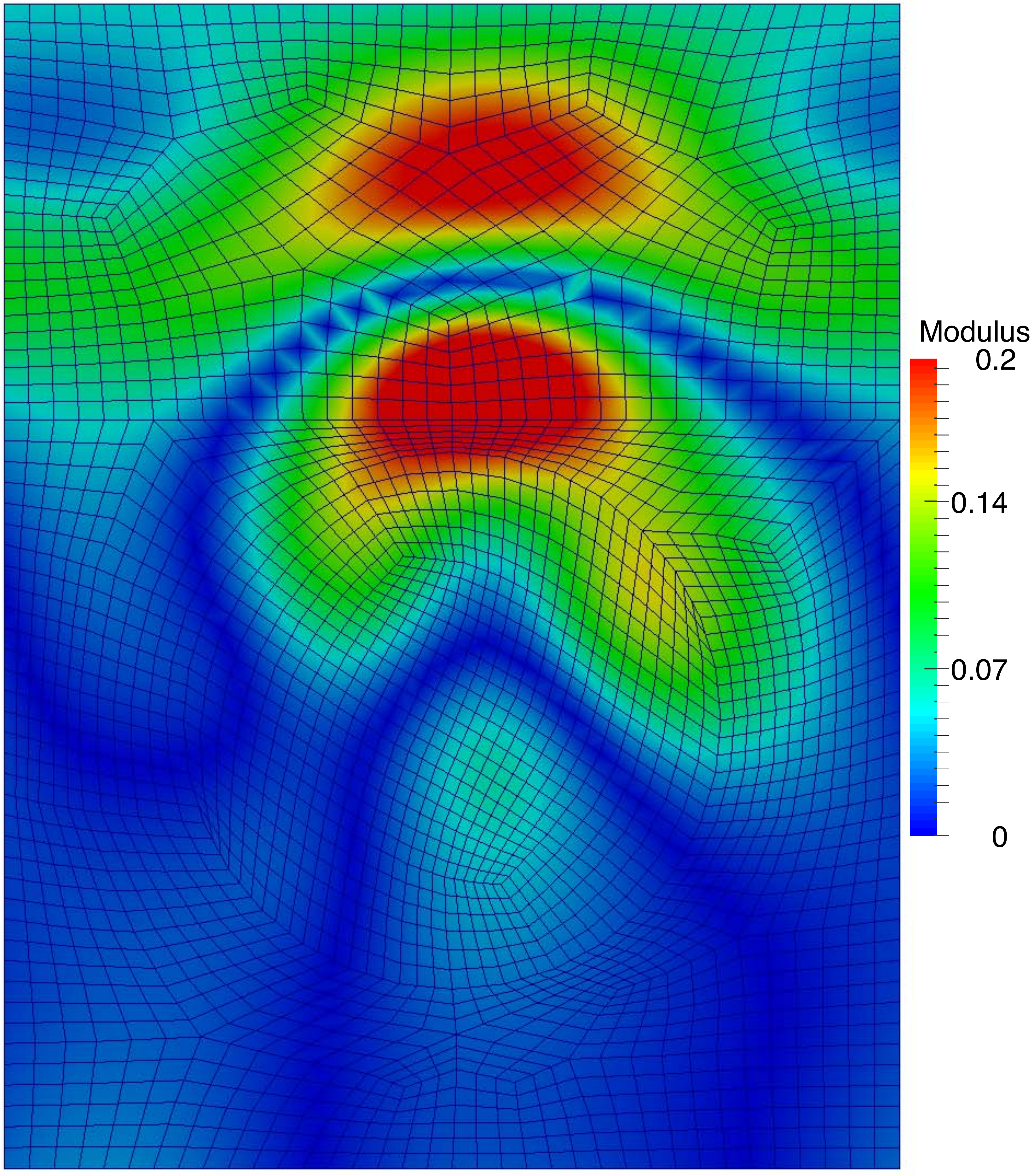}}\\
\resizebox{0.42\textwidth}{!}{\includegraphics{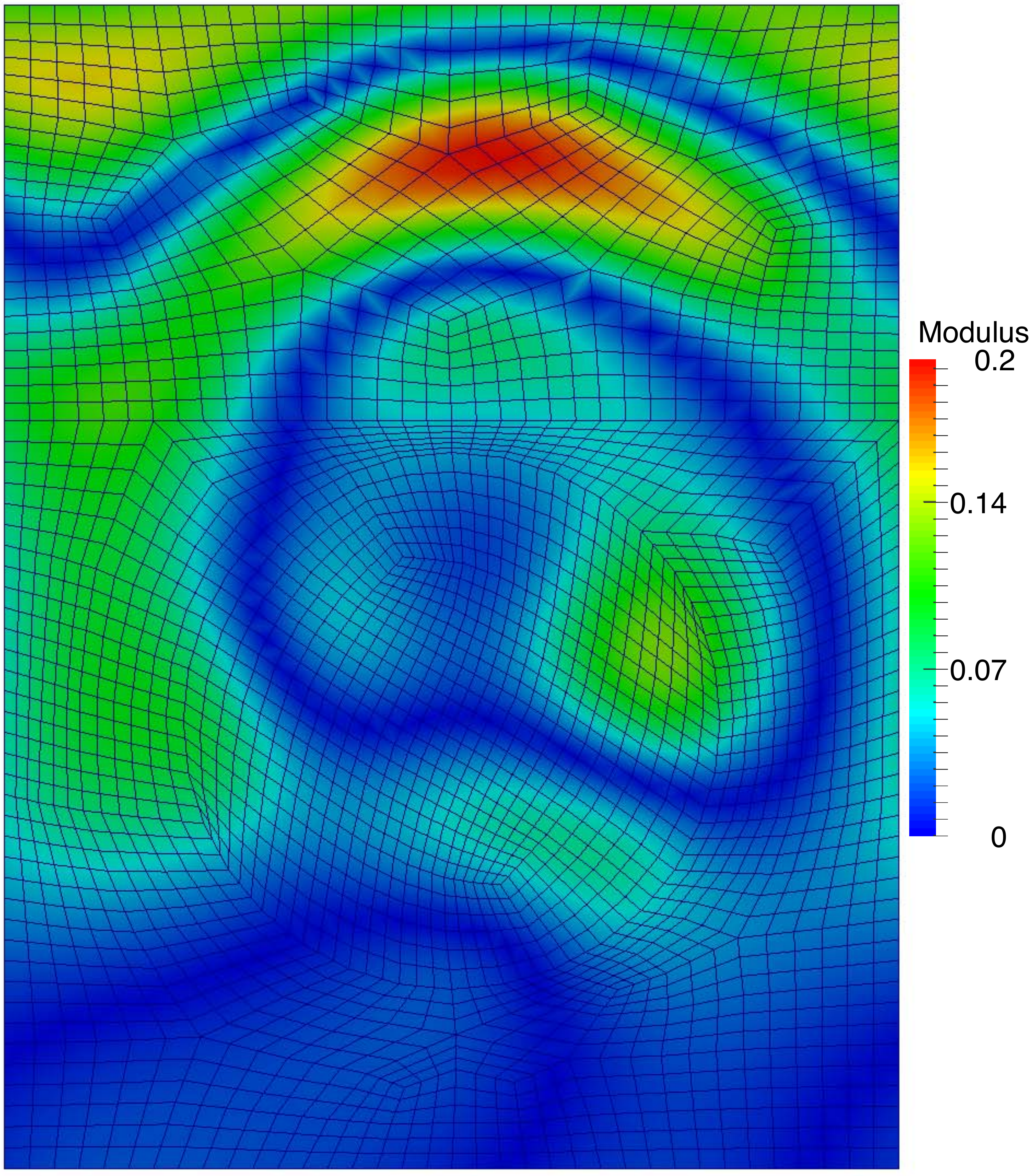}}&
\resizebox{0.42\textwidth}{!}{\includegraphics{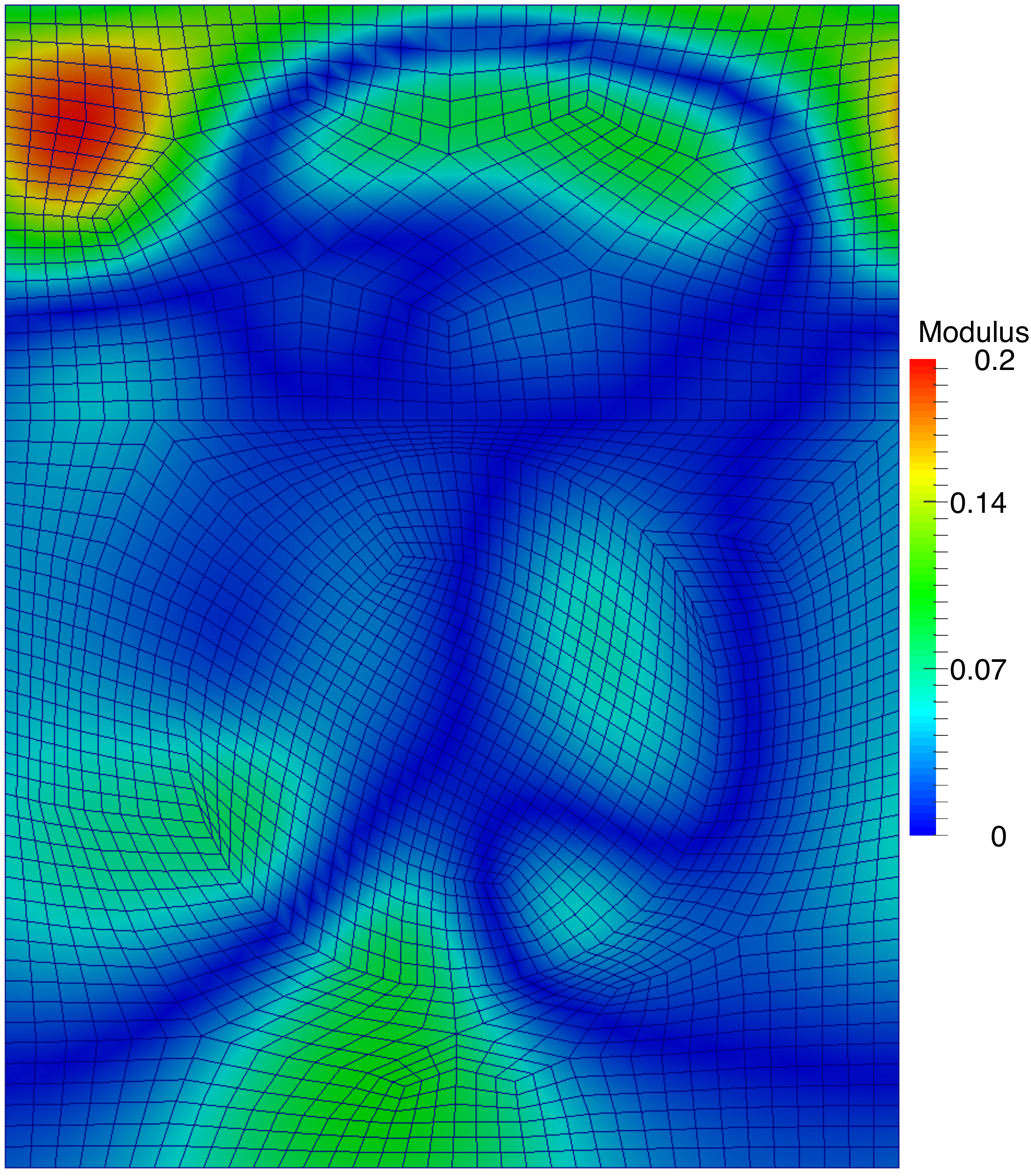}}
\end{tabular}
\end{center}
\caption{Snapshots of density plots of the absolute value of  the total field $w(\bfx,t)$ for the Sellmeier model
at times $ t=3.59,4.44,5.29,6.14,6.99,7.84$ (from top left to bottom right). }
\label{sellsnap}
\end{figure}

\section{Conclusion}\label{concl}
In this paper we have proved a general existence and continuous dependence result for time domain solutions of the grating problem with frequency dependent coefficients.  We have shown that this leads to a convergent time discretization.  Following spatial discretization with a typical non-overlapping finite element and spectral
domain decomposition technique we have verified the time stepping convergence rate, as well as demonstrating the
apparently stable numerical solution of two problems involving typical frequency dependent material 
coefficients.

The study needs to be completed by an analysis of spatial discretization including mesh truncation and this will be the subject of a future paper.  Further testing of the approach on real solar voltaic devices would also reveal the benefits and limitations of the time domain approach. \red{In particular we need to investigate the relative efficiency of time domain and frequency domain approaches.}

\section*{Acknowledgements}
The research of  L. Fan is supported by NSF grant number  DMS-1125590.  The research of P. Monk is supported in part by NSF grants DMS-1216620 and DMS-1125590.

\newpage
\bibliographystyle{elsarticle-num}
\bibliography{../D_t_N/grating.bib}

\end{document}